\documentclass[11pt]{amsart}

\usepackage[T1]{fontenc}
\usepackage[utf8]{inputenc}
\usepackage{lmodern}
\usepackage{amsmath,amssymb,amsthm,mathtools}
\usepackage{mathrsfs}
\usepackage{enumitem}
\usepackage{graphicx}
\usepackage{xcolor}
\usepackage{hyperref}
\usepackage[nameinlink,capitalize,noabbrev]{cleveref}
\usepackage{booktabs}
\usepackage{array}
\usepackage{longtable}
\usepackage{caption}
\usepackage{subcaption}
\usepackage{microtype}
\usepackage[section]{placeins}
\allowdisplaybreaks
\graphicspath{{figures/}}

\hypersetup{
  colorlinks=true,
  linkcolor=blue!50!black,
  citecolor=blue!50!black,
  urlcolor=blue!50!black,
  pdftitle={The Partition Graph as a Growing Discrete Geometric Object},
  pdfauthor={Fedor B. Lyudogovskiy}
}

\theoremstyle{plain}
\newtheorem{theorem}{Theorem}[section]
\newtheorem{proposition}[theorem]{Proposition}
\newtheorem{lemma}[theorem]{Lemma}

\newtheorem{problem}[theorem]{Problem}
\newtheorem{question}[theorem]{Question}
\newtheorem{heuristic}[theorem]{Heuristic Principle}

\theoremstyle{definition}
\newtheorem{definition}[theorem]{Definition}
\newtheorem{remark}[theorem]{Remark}

\newtheorem{observation}[theorem]{Observation}

\newcommand{\Cl}{\mathrm{Cl}}
\newcommand{\SC}{\mathrm{SC}}
\newcommand{\Spine}{\mathrm{Sp}}
\newcommand{\dist}{\operatorname{dist}}
\newcommand{\locdim}{\dim_{\mathrm{loc}}}
\newcommand{\deglayer}{D}
\newcommand{\simpLayer}{L}
\newcommand{\axdist}{d_{\mathrm{ax}}}

\newcommand{\Gn}{G_n}
\newcommand{\Kn}{K_n}
\newcommand{\Mn}{M_n}
\newcommand{\Ln}{L_n}
\newcommand{\Rn}{R_n}
\newcommand{\Fn}{F_n}
\newcommand{\Cn}[1]{\mathcal{C}_n^{(#1)}}

\newcommand{\glossentry}[2]{%
  \par\medskip
  \noindent\textbf{#1}\par
  #2\par
}

\title{The Partition Graph as a Growing Discrete Geometric Object}
\author{Fedor B. Lyudogovskiy}

\begin{document}

\begin{abstract}
For each positive integer \(n\), let \(\Gn\) be the graph of integer partitions of \(n\), where two partitions are adjacent if one is obtained from the other by an elementary transfer of a cell in the Ferrers diagram, followed by reordering. Previous work has studied the global homotopy type of the clique complex \(\Kn=\Cl(\Gn)\) and the local combinatorics of the graph at a fixed vertex. This paper initiates the study of \(\Gn\) itself as a growing discrete geometric object.

Primarily foundational in character, the paper introduces a structural language for the large-scale morphology of partition graphs, including the antenna vertices, main chain, boundary framework, self-conjugate axis, simplex layers, degree landscape, central region, and spine. Using local invariants from the companion local theory---especially degree and local simplex dimension---it defines canonical vertex layerings of \(\Gn\).

A small computational atlas for \(1\le n\le 12\) illustrates the first visible emergence of framework, interior, axial organization, and central skeletal structure. The paper provides a vocabulary, a family of canonical constructions, and a first small-range developmental picture for later quantitative, asymptotic, and theorem-centered work.
\end{abstract}

\maketitle

\section{Introduction}

Graphs on integer partitions arise when partitions are viewed not as isolated combinatorial objects but as vertices connected by elementary transformations. In the partition graph \(\Gn\), the vertices are the partitions of \(n\), and two vertices are adjacent when one partition is obtained from the other by a single elementary transfer of a cell in the Ferrers diagram, followed by reordering. This graph model is natural from several points of view. It is closely related to classical transfer operations on partitions and to minimal-change constructions in the theory of combinatorial Gray codes, but it also has intrinsic geometric interest of its own: it endows the set of partitions of \(n\) with the structure of a finite graph that exhibits visible internal organization \cite{LyuHomotopy,LyuLocal,Savage,RasmussenSavageWest,Mutze}.

Two scales of structure in this graph have already been studied. In one direction, the clique complex
\[
\Kn:=\Cl(\Gn)
\]
was analyzed from a global topological point of view, and it was shown that \(\Kn\) is always homotopy equivalent to a wedge of \(2\)-spheres \cite{LyuHomotopy}. That work also highlighted a striking qualitative contrast: although the local simplex structure of \(\Kn\) becomes arbitrarily complicated as \(n\) grows, its global homotopy type remains that of a wedge of \(2\)-spheres \cite{LyuHomotopy}.

In another direction, the local combinatorics of \(\Gn\) at a fixed vertex \(\lambda\) was described in terms of admissible transfers, a canonical local bipartite graph \(B(\lambda)\), the induced neighborhood graph, degree, local clique structure, and local simplex dimension. In particular, these local invariants were shown to depend only on a simple ordered binary datum attached to the support pattern of the partition. The companion local paper explicitly presents this as the first local layer of a broader program and points toward later study of larger-scale structures, central regions, core-like subgraphs, and growth phenomena as \(n\) varies \cite{LyuLocal}.

This paper concerns a third scale, intermediate between those two viewpoints. It studies neither the global homotopy type of \(\Kn\) nor the local structure at a single vertex, but the graph \(\Gn\) itself as a growing discrete geometric object. Its goal is to describe the large-scale and mesoscopic morphology of partition graphs: their extremal framework, symmetry axis, local-complexity layerings, central body, and axial skeletal structure.

A simple observation motivates this viewpoint: although \(\Gn\) is defined combinatorially, it has clear internal structure. For \(n\ge 2\) the graph has two distinguished degree-one vertices, called the antenna vertices; for \(n=1\), the unique vertex plays the corresponding degenerate extremal role. They are joined by a canonical shortest hook-like path, the main chain. From this chain branch the two lateral extremal paths, called the left and right edges.

The graph also admits a natural involutive symmetry given by conjugation of partitions; its fixed-point set is the self-conjugate axis. Around that axis, as \(n\) grows, one begins to see the emergence of a central region, a richer local simplex structure, a more varied degree landscape, and an axial scaffold called the spine.

The paper is primarily foundational and descriptive rather than exhaustive. It does not attempt to settle later questions about degree growth, support phenomena, jump-type invariants, centrality, or asymptotic behavior. Instead, it introduces a structural language in which such questions can be formulated naturally. In particular, it uses local invariants imported from the companion local theory---especially degree and local simplex dimension---to define canonical vertex layerings of \(\Gn\), and a small computational atlas to examine the early emergence of the outer framework, the interior body, the axial center, and the spine.

Accordingly, the family
\[
G_1,G_2,G_3,\dots
\]
should be viewed not only as a sequence of finite graphs, but also as a family of increasingly articulated discrete forms. The paper adopts an intentionally morphological vocabulary. Terms such as antenna vertices, framework, axis, central region, and spine are used not as loose metaphors, but as names for graph-theoretically defined or graph-theoretically motivated structures that organize the large-scale geometry of partition graphs.

The paper serves as a foundation for a broader program on the geometry and morphology of partition graphs. It is partly formal and partly exploratory: exact graph-theoretic definitions and basic propositions are combined with a small computational atlas and with open structural questions for later work. The emphasis is not on a single new theorem, but on a coordinated family of definitions, constructions, and observations.

The paper is organized as follows. Section~2 recalls the partition graph and introduces the antenna vertices, the main chain, the edges, and the boundary framework. Section~3 studies conjugation symmetry and the self-conjugate axis. Section~4 uses local simplex dimension to define simplex layers and a first morphological organization. Section~5 turns to the degree landscape. Section~6 introduces axial central regions and the spine. Section~7 discusses directionality and anisotropy. Section~8 adopts a morphogenetic viewpoint on the family \((G_n)\). Section~9 presents an atlas of small partition graphs. Section~10 formulates heuristic principles, open questions, and future directions, and Section~11 concludes the paper. Appendix~A records the small-range numerical data, and Appendix~B collects the glossary of structural terms.

\section{The partition graph as a geometric object}

\subsection{The graph \texorpdfstring{\(\Gn\)}{G\_n}}

A partition of a positive integer \(n\) is a finite weakly decreasing sequence
\[
\lambda=(\lambda_1,\lambda_2,\dots,\lambda_\ell), \qquad
\lambda_1\ge \lambda_2\ge \dots \ge \lambda_\ell>0, \qquad
\sum_{i=1}^{\ell}\lambda_i=n.
\]
We write \(\lambda\vdash n\) when \(\lambda\) is a partition of \(n\). As in \cite{LyuHomotopy,LyuLocal}, we identify a partition with its Ferrers diagram and freely pass between row-length notation and diagrammatic language.

The partition graph \(\Gn\) is the graph whose vertices are the partitions of \(n\), and whose edges correspond to elementary transfers of one cell in the Ferrers diagram: one moves a cell from a removable corner to an addable corner, followed by reordering, provided that the resulting partition is different from the original one. Equivalently, an edge corresponds to transferring one cell from one row to another row, where the target row may also be a new row of length \(0\), thereby creating a new part \(1\). This is exactly the graph model studied in \cite{LyuHomotopy,LyuLocal}.

\begin{lemma}\label{lem:connected}
For every \(n\ge 1\), the graph \(\Gn\) is connected.
\end{lemma}

\begin{proof}
Let \(\lambda=(\lambda_1,\dots,\lambda_\ell)\vdash n\). If \(\ell=1\), then \(\lambda=(n)\). Otherwise, \(\ell\ge 2\), so the last row contains a removable corner and the first row admits an addable corner. Transferring one cell from the last row to the first row yields, after reordering, a neighbor of \(\lambda\) in \(\Gn\). Repeating this operation eventually absorbs all rows into the first row and reaches \((n)\). Hence every vertex is connected to \((n)\), so \(\Gn\) is connected.
\end{proof}

\subsection{Ferrers-diagram viewpoint}

The Ferrers-diagram model gives \(\Gn\) much more structure than a generic finite graph would carry. First, an edge has a geometric interpretation: it is an elementary motion of one cell. Second, conjugation of partitions acts naturally on the vertex set and preserves adjacency. Third, certain subfamilies of partitions, such as hook partitions, two-part partitions, self-conjugate partitions, and rectangular partitions, occupy visibly distinguished positions. The graph thus inherits a nontrivial internal geometry from the geometry of partitions themselves.

A closer inspection of the partition graph \(\Gn\) shows that its outer morphology is not organized by a single boundary path, but by a small system of distinguished extremal structures. At the most basic level, these are the extremal hook endpoints, the shortest path connecting them, and the side chains branching from that path. Together they form the first large-scale outer framework of the graph.

\subsection{The antenna vertices}

\begin{definition}
The vertices
\[
(n), \qquad (1^n)
\]
are called the \emph{antenna vertices} of \(\Gn\). For \(n\ge 2\) they are distinct degree-one vertices; for \(n=1\) they coincide and form a degenerate one-vertex version of the same notion. The term is motivated by the two outer antenna-like protrusions visible in the atlas of small partition graphs.
\end{definition}

They are exchanged by conjugation and determine the maximal longitudinal extent of the partition graph.

\begin{proposition}\label{prop:antennas-degree-one}
For \(n\ge 2\), the only degree-one vertices of \(\Gn\) are the two antenna vertices \((n)\) and \((1^n)\).
\end{proposition}

\begin{proof}
The partition \((n)\) is adjacent only to \((n-1,1)\), and by conjugation \((1^n)\) is adjacent only to \((2,1^{n-2})\). Conversely, let \(\lambda\vdash n\) be different from both \((n)\) and \((1^n)\). Then \(\lambda\) has at least two parts and also satisfies \(\lambda_1\ge 2\). Transferring one cell from the last row to the first row produces a neighbor \(\mu\) of \(\lambda\). Transferring one cell from the first row to a new row of length \(1\) produces another neighbor \(\nu\) of \(\lambda\). These neighbors are distinct because \(\nu\) has \(\ell+1\) parts, whereas \(\mu\) has at most \(\ell\) parts. Hence \(\deg_{\Gn}(\lambda)\ge 2\).
\end{proof}

\subsection{The main chain}

For \(n\ge 2\), between the two antenna vertices there is a canonical path formed by the hook partitions
\[
(n),\ (n-1,1),\ (n-2,1^2),\ \dots,\ (n-k,1^k),\ \dots,\ (2,1^{n-2}),\ (1^n).
\]

\begin{definition}
The \emph{main chain} of \(\Gn\) is the hook-partition path
\[
\Mn:=\{(n-k,1^k):0\le k\le n-1\}.
\]
\end{definition}

\begin{proposition}
For \(n\ge 2\), the vertices of \(\Mn\), taken in their natural order, form a shortest path in \(\Gn\) between the two antenna vertices. For \(n=1\), \(\Mn\) consists of the single vertex \((1)\).
\end{proposition}

\begin{proof}
Each elementary transfer changes the number of parts by at most one. The partition \((n)\) has exactly one part, whereas \((1^n)\) has exactly \(n\) parts. Hence every path from \((n)\) to \((1^n)\) in \(\Gn\) has length at least \(n-1\). The hook path \(\Mn\) has exactly \(n\) vertices and therefore length \(n-1\), so it realizes this lower bound. The case \(n=1\) is immediate.
\end{proof}

The second vertices from the two ends of the main chain are
\[
(n-1,1)\qquad\text{and}\qquad (2,1^{n-2}).
\]
For \(n\ge 4\), these are the first branching points of the outer framework. Indeed, \((n-1,1)\) is adjacent exactly to \((n)\), \((n-2,1^2)\), and \((n-2,2)\), while \((2,1^{n-2})\) is its conjugate.

\subsection{The left and right edges}

On one side, we have the chain of two-part partitions
\[
(n-1,1),\ (n-2,2),\ (n-3,3),\ \dots,\ \bigl(n-\lfloor n/2\rfloor,\lfloor n/2\rfloor\bigr).
\]

\begin{definition}
The \emph{left edge} of \(\Gn\) is
\[
\Ln:=\{(n-k,k):1\le k\le \lfloor n/2\rfloor\}.
\]
The \emph{right edge} of \(\Gn\) is
\[
\Rn:=\{\lambda' : \lambda\in \Ln\}.
\]
\end{definition}

\begin{proposition}
The sets \(\Ln\) and \(\Rn\) form paths in \(\Gn\). Moreover, \(\Ln\) meets \(\Mn\) at \((n-1,1)\), and \(\Rn\) meets \(\Mn\) at \((2,1^{n-2})\).
\end{proposition}

\begin{proof}
Successive vertices of \(\Ln\) differ by transferring one cell from the first part to the second part, so \(\Ln\) is a path in \(\Gn\). Since conjugation is a graph automorphism (Proposition~\ref{prop:conjugation-automorphism} below), the image \(\Rn\) is also a path.

Every vertex of \(\Ln\) corresponds to a partition with exactly two parts. A hook partition has the form \(n-k,1^k\), so among two-part partitions only \((n-1,1)\) is a hook. Hence \(\Ln\cap \Mn=\{(n-1,1)\}\). Conjugating gives \(\Rn\cap \Mn=\{(2,1^{n-2})\}\).
\end{proof}

\subsection{Boundary framework}

\begin{definition}
The \emph{boundary framework} of \(\Gn\) is
\[
\Fn:=\Mn\cup \Ln\cup \Rn.
\]
A vertex of \(\Gn\) is called an \emph{interior vertex} if it does not belong to \(\Fn\).
\end{definition}

The boundary framework should be understood as a first canonical outer scaffold of \(\Gn\), not as an exhaustive description of every extremal feature visible in the atlas. This terminology reflects both visual and combinatorial structure. For \(n\ge 2\), the main chain provides the lower longitudinal contour connecting the two antenna vertices. The left and right edges branch from it near the two ends and trace the lateral contours of the graph. In the small-graph atlas, this framework already captures a substantial part of the visible outer morphology.

The conjugation symmetry acts naturally on this framework. It exchanges the two antenna vertices, reverses the main chain as a set, and swaps the left and right edges. Thus the boundary framework is a canonical part of the large-scale geometry of \(\Gn\).

At this stage, it is useful to distinguish the framework from the deeper interior of the graph. The interior is the set
\[
V(\Gn)\setminus \Fn.
\]

\begin{remark}
The atlas also suggests a further extremal structure associated with rectangular partitions. We do not formalize it here, but record it as a candidate for a future rectangular contour or rear edge complementing the present outer scaffold.
\end{remark}

\section{Conjugation symmetry and the self-conjugate axis}

The partition graph \(\Gn\) carries a natural involutive symmetry given by conjugation of partitions. At the level of Ferrers diagrams, conjugation interchanges rows and columns and preserves the elementary transfer relation. Thus conjugation defines a graph automorphism
\[
\lambda \longmapsto \lambda'
\]
of \(\Gn\).

This symmetry is one of the main organizing principles of the large-scale morphology of the partition graph. It exchanges the two antenna vertices
\[
(n)\longleftrightarrow (1^n),
\]
reverses the main chain
\[
(n),\ (n-1,1),\ (n-2,1^2),\ \dots,\ (2,1^{n-2}),\ (1^n),
\]
and swaps the left and right edges. Thus the boundary framework introduced above is not just a convenient collection of extremal paths: it is complemented by a canonical internal symmetry.

\subsection{Conjugation as an automorphism}

\begin{proposition}\label{prop:conjugation-automorphism}
Conjugation of partitions defines an involutive automorphism of \(\Gn\).
\end{proposition}

\begin{proof}
Conjugation interchanges rows and columns in Ferrers diagrams, sending removable corners to removable corners and addable corners to addable corners. Accordingly, an elementary transfer of one cell in a Ferrers diagram becomes an elementary transfer of one cell in the conjugate diagram. Hence adjacency in \(\Gn\) is preserved. Since applying conjugation twice returns the original partition, the map is involutive.
\end{proof}

\subsection{The self-conjugate axis}

\begin{definition}
The \emph{self-conjugate axis} of \(\Gn\) is
\[
\SC_n:=\{\lambda\vdash n:\lambda=\lambda'\}.
\]
\end{definition}

Here ``axis'' is meant in a structural sense. In general, \(\SC_n\) need not form a single path or chain inside the graph. Its role is different: it is the fixed-point set of the distinguished involution of \(\Gn\). As such, it provides a canonical central reference set for the geometry of the graph. In the visual atlas of small partition graphs, the self-conjugate vertices occupy a visibly central position and organize the surrounding morphology.

Since \(\Gn\) is connected by Lemma~\ref{lem:connected}, the distance from a vertex to \(\SC_n\) is an ordinary finite graph distance whenever \(\SC_n\neq\varnothing\). This leads to the notion of axial distance, which will be used below to define central regions and to describe the emergence of a spine-like axial structure.

The interaction between conjugation symmetry and the boundary framework is especially important. The two antenna vertices are exchanged by conjugation, the main chain is invariant as a set, and the two edges are mapped to one another. Thus the outer shape of \(\Gn\) is reflected by an internal axial symmetry. This gives the graph a more rigid large-scale organization than one would expect in a generic combinatorial graph.

\subsection{Axial distance}

\begin{definition}
For a subset \(S\subseteq V(\Gn)\), define
\[
\dist_{\Gn}(\lambda,S):=
\begin{cases}
\min_{\mu\in S}\dist_{\Gn}(\lambda,\mu), & S\neq\varnothing,\\
\infty, & S=\varnothing.
\end{cases}
\]
For \(\lambda\in V(\Gn)\), the \emph{axial distance} is then defined by
\[
\axdist(\lambda):=\dist_{\Gn}(\lambda,\SC_n).
\]
\end{definition}

Axial distance provides a first bridge between symmetry and morphology. It allows one to distinguish vertices lying close to the central symmetric zone from those located farther away toward the outer framework. In particular, it will be used below to define central regions and to describe the emergence of a spine-like axial structure.

From a large-scale point of view, the role of the self-conjugate axis is twofold. First, it is the canonical center of symmetry of \(\Gn\). Second, it acts as a geometric organizer: both the degree landscape and the local simplex layering tend to become richer in its vicinity than on the extreme outer framework. This contrast between symmetric center and sparse periphery is one of the most basic morphological features of the family of partition graphs.

\section{Local simplex layers and morphological organization}

\subsection{Local input from the companion paper}

We briefly recall the relevant results from the local theory. For a partition
\[
\lambda=(s_1^{m_1},\dots,s_t^{m_t}),
\qquad s_1>\cdots>s_t>0,
\]
the ordered local transfer type
\[
t(\lambda)=(t;\alpha_1,\dots,\alpha_t;\beta_1,\dots,\beta_t)
\]
records the support size \(t\), the singleton-block indicators \(\alpha_i\), and the unit-gap indicators \(\beta_i\). The main local structure theorem of \cite{LyuLocal} shows that this ordered binary datum completely determines the local admissibility graph \(B(\lambda)\), the induced neighborhood graph \(N(\lambda)\), the degree of \(\lambda\), and the local clique number \(\omega_{\mathrm{loc}}(\lambda)\). It also determines the local simplex dimension \(\locdim(\lambda)\), that is, the largest simplex dimension of \(\Kn=\Cl(\Gn)\) among simplices containing \(\lambda\). In particular, one has
\[
N(\lambda)\cong L(B(\lambda)),
\]
and both degree and local simplex dimension admit explicit descriptions in terms of \(t(\lambda)\).

For the purposes of the paper, the most important consequence is that the local complexity of a vertex is encoded by a small intrinsic combinatorial signature. Thus local morphology can be compared across distant regions of the graph without referring to any ambient drawing.

\subsection{Simplex layers}

We now introduce the simplest global vertex layering derived from this local theory.

\begin{definition}
For \(r\ge 0\), define
\[
\simpLayer_r(n):=\{\lambda\in V(\Gn):\locdim(\lambda)=r\}.
\]
We call \(\simpLayer_r(n)\) the \emph{\(r\)-th simplex layer} of \(\Gn\).
\end{definition}

Thus the vertex set of \(\Gn\) is partitioned according to the largest simplex dimension among simplices containing a given vertex. Equivalently, the simplex layers record the local clique richness of vertices in \(\Kn=\Cl(\Gn)\).

This layering is natural for several reasons. First, it is intrinsic: it depends only on the local combinatorics of \(\Gn\). Second, it is compatible with the earlier topological point of view, since \(\locdim(\lambda)\) is defined in terms of simplices in \(\Kn\). Third, it is morphologically informative: unlike purely formal labels, the simplex layers can be visualized on the graph and reveal a non-uniform spatial distribution.

The smallest simplex layers occur at the most rigid vertices. For example, the two antenna vertices \((n)\) and \((1^n)\) have degree \(1\). The largest simplex of \(K_n\) containing either of them is just the edge joining it to its unique neighbor; hence their local clique number is \(2\) and their local simplex dimension is \(1\). More generally, vertices lying on the outer framework tend to support simpler local configurations than vertices in the central body. By contrast, larger values of \(\locdim\) arise where several admissible transfers coexist in a more highly organized way. This already suggests that simplex layers should be viewed as a measure of local thickness.

The local structure theorem also shows that simplex complexity is not an arbitrary function on partitions. Since \(\locdim(\lambda)\) depends only on the ordered local transfer type \(t(\lambda)\), vertices with the same local transfer type belong automatically to the same simplex layer. In this sense, the simplex layering is coarser than the classification by local transfer type, but finer than crude graph-theoretic distinctions such as degree alone.

\subsection{Morphological meaning}

From the large-scale morphological point of view, the main significance of the simplex layers is that they make visible the transition from sparse outer regions to richer inner zones. In the small-graph atlas, low simplex layers dominate the antenna vertices, the main chain, and much of the outer framework, while higher layers begin to accumulate nearer the self-conjugate axis and within the central body. Thus the simplex layers provide a first precise sense in which the graph develops an interior of increasing local complexity.

This point is especially interesting in light of the earlier topological result on the clique complex \(\Kn\). Although the local simplex structure becomes more complicated as \(n\) grows, the global homotopy type of \(\Kn\) remains extremely simple: it is always a wedge of \(2\)-spheres \cite{LyuHomotopy}. The simplex layers belong exactly to the side of the theory where complexity accumulates locally rather than globally.

The simplex layers are only the first and most geometric of several possible local layerings. One may likewise partition \(V(\Gn)\) by degree, by local transfer type, or by isomorphism class of the neighborhood graph \(N(\lambda)\). Among these, however, the simplex layering is especially well suited to the present paper, because it links local combinatorics directly to the emerging geometry of the graph and to the earlier clique-complex viewpoint.

\subsection{Degree layers}

For later use, it is also convenient to introduce a parallel degree layering.

\begin{definition}
For \(d\ge 0\), define
\[
\deglayer_d(n):=\{\lambda\in V(\Gn):\deg_{\Gn}(\lambda)=d\}.
\]
We call \(\deglayer_d(n)\) the \emph{degree-\(d\) layer} of \(\Gn\).
\end{definition}

The degree layers and the simplex layers are related but not identical. Degree measures the total number of admissible directions from a vertex, whereas \(\locdim\) measures the size of the largest coherent simplex through that vertex. Thus degree reflects local abundance of moves, while simplex dimension reflects local compatibility of moves. Their comparison will play a role in the description of the degree landscape in the next section.

At the descriptive level relevant here, the main lesson is that \(\Gn\) is not morphologically uniform. Its vertices fall into natural local-complexity classes, and these classes exhibit a visible spatial organization. The simplex layers provide the first systematic language for this organization. They allow us to speak not only about central and peripheral regions in an intuitive way, but also about zones in which different levels of local simplex complexity occur and concentrate.

Simplex layering is the first real bridge between the local theory of a single vertex and the large-scale theory of the whole graph. It does not describe the complete geometry of \(\Gn\), but it does provide a usable coordinate for that geometry.

\section{Degree landscape}

The simplex layers introduced in the previous section measure one aspect of local complexity: the size of the largest coherent simplex through a vertex. A second, more immediate local invariant is the degree. While simplex dimension measures the maximal compatibility of admissible moves, the degree counts their total number. Taken over the whole graph, this gives rise to what may be called the degree landscape of \(\Gn\).

\subsection{Degree as local mobility}

Recall from the local structure theorem of \cite{LyuLocal} that the degree of a partition \(\lambda\) is completely determined by its ordered local transfer type \(t(\lambda)\). In particular, degree is not an arbitrary function on the graph: it is encoded by the block structure of the partition itself. Thus, just as the simplex layers partition \(V(\Gn)\) according to local clique richness, the degree layers partition \(V(\Gn)\) according to local mobility under elementary transfers.

For \(d\ge 0\), we have already defined
\[
\deglayer_d(n):=\{\lambda\in V(\Gn):\deg_{\Gn}(\lambda)=d\},
\]
which is self-conjugate under conjugation. At one extreme of this landscape lie the antenna vertices
\[
(n),\qquad (1^n),
\]
which have degree \(1\). They are the unique vertices of minimal mobility and therefore mark the two extreme ends of the graph. The next vertices on the main chain,
\[
(n-1,1)\qquad\text{and}\qquad (2,1^{n-2}),
\]
already have higher degree; for \(n\ge 4\), they are the first branching points of the outer framework and have degree \(3\). This is precisely where the left and right edges split off from the main chain. Thus the degree landscape already reflects the basic longitudinal structure of \(\Gn\): the antenna vertices are terminal points, while the adjacent hook vertices are the first nontrivial branching points.

More generally, the outer framework tends to support relatively small degrees. This does not mean that every framework vertex is locally trivial. For instance, a rectangular partition \((r^m)\) with \(r,m\ge 2\) has degree \(2\), yet its two neighbors are adjacent, so it already belongs to a triangle in \(\Gn\), and its local simplex dimension is \(2\). Thus low degree does not imply complete local simplicity. What it does indicate is a restricted number of admissible directions.

At the opposite end of the landscape lie vertices with many admissible transfers. The local theory shows that high degree can occur even in the presence of many forbidden local edges, provided the support structure is sufficiently rich. A particularly instructive example is the staircase partition
\[
\delta_t=(t,t-1,\dots,2,1),
\]
which exists exactly when \(n=t(t+1)/2\) is a triangular number, is self-conjugate, and satisfies
\[
\deg_{\Gn}(\delta_t)=t(t-1),\qquad \locdim(\delta_t)=t-1
\]
by the formulas from \cite{LyuLocal}. Thus degree may already grow linearly in \(n\) along highly structured central configurations, even when local admissibility is constrained by many singleton blocks and unit gaps.

\subsection{Degree versus simplex dimension}

This already shows that the degree landscape is not simply a blurred copy of the simplex layering. The two are related, but they measure different aspects of local geometry. Degree counts admissible moves; simplex dimension measures the largest family of mutually compatible moves. In particular, two vertices may have comparable degree but different simplex dimension, or vice versa. The two functions should be regarded as complementary coordinates on the local morphology of \(\Gn\).

From the large-scale point of view, the most important feature of the degree landscape is its spatial non-uniformity. In the computed atlas of small partition graphs, low degrees dominate the outer framework, while larger degrees accumulate toward the interior and especially near the self-conjugate axis. Degree thus acts as a first quantitative marker of the distinction between sparse periphery and richer central body.

This tendency becomes clearer as \(n\) grows. New degree values appear, and the range of the degree function widens. The resulting landscape is not flat: it develops visible relief. Some regions remain narrow and constrained, while others support increasingly many admissible directions. The partition graph thus begins to look less like a homogeneous discrete network and more like a structured body with zones of different local mobility.

It is useful to emphasize that this phenomenon is fully compatible with the local transfer-type theory. Since degree depends only on \(t(\lambda)\), the emergence of a richer degree landscape does not reflect disorder. On the contrary, it reflects the increasing diversity of local support patterns available as \(n\) grows. The global morphology becomes richer because the repertoire of local transfer types becomes richer.

\subsection{Symmetry of the degree landscape}

The degree landscape also interacts naturally with the axial picture introduced earlier. Since conjugation preserves the graph structure, it preserves degree as well.

\begin{proposition}
Conjugation preserves degree:
\[
\deg_{\Gn}(\lambda)=\deg_{\Gn}(\lambda')
\]
for every vertex \(\lambda\in V(\Gn)\).
\end{proposition}

\begin{proof}
Conjugation is a graph automorphism of \(\Gn\), so it preserves adjacency and therefore the degree of each vertex.
\end{proof}

Thus degree is one of the simplest quantitative fields on \(\Gn\) compatible with the canonical symmetry of the graph.

For the purposes of the present paper, the main lesson is descriptive rather than classificatory. Degree should be viewed not only as a local invariant of individual vertices, but also as a global landscape on the whole graph. This landscape already captures one of the central morphological facts about partition graphs: the outer framework is rigid and low-mobility, whereas the interior develops richer and richer local dynamics.

Once degree is regarded as a spatial field rather than merely as a formula, it becomes natural to ask where its higher values concentrate, how they relate to simplex layers, and how they interact with the central region and the spine. These are precisely the structures examined in the next section.

\section{Central region and spine}

The outer morphology of the partition graph \(\Gn\) is organized by the boundary framework
\[
\Fn:=\Mn\cup \Ln\cup \Rn,
\]
consisting of the main chain together with the left and right edges. At the same time, the graph carries a canonical internal symmetry given by conjugation, whose fixed-point set is the self-conjugate axis
\[
\SC_n:=\{\lambda\vdash n:\lambda=\lambda'\}.
\]
The purpose of the present section is to define the first genuinely internal morphological structures of the graph: the central region and the spine.

These notions are meant to capture the fact, already visible in the atlas of small partition graphs, that the interior of \(\Gn\) is not amorphous. As \(n\) grows, a visibly central body emerges around the self-conjugate axis. This body appears visually separated from the outer framework and tends to support richer local structure than the antenna vertices and the lateral extremities. In particular, higher degrees and higher simplex layers tend to accumulate there.

\subsection{Axial central regions}

We begin with the graph-theoretic notion of axial distance. For a vertex \(\lambda\in V(\Gn)\), define
\[
\axdist(\lambda):=\dist_{\Gn}(\lambda,\SC_n).
\]
Thus \(\axdist(\lambda)\) measures how far \(\lambda\) lies from the self-conjugate axis in graph distance.

Using this quantity, we define a family of axial central regions.

\begin{definition}
For \(r\ge 0\), the \emph{axial \(r\)-central region} of \(\Gn\) is
\[
\Cn{r}:=\{\lambda\in V(\Gn)\setminus \Fn:\axdist(\lambda)\le r\}.
\]
\end{definition}

\begin{definition}
The \emph{narrow central region} is \(\Cn{1}\).
\end{definition}

This definition isolates the part of the graph that lies away from the canonical outer framework while remaining close to the conjugation axis. The parameter \(r\) allows one to pass from a very thin central body to thicker axial zones. In this paper, the cases \(r=1\) and \(r=2\) are the most relevant: \(\Cn{1}\) captures the narrow central body, while \(\Cn{2}\) provides a useful first thickening.

The exclusion of \(\Fn\) is intentional. The central regions are not meant to be metric balls around the self-conjugate axis. Rather, they are designed to isolate the interior axial body after the rigid outer scaffold formed by the main chain and the two edges has been removed.

It is important to note that the present definition is intentionally conservative. The set \(\Fn\) includes the main chain and the two edges, but does not yet include the prospective rectangular contour. Since that rear structure has not yet been formalized, the current notion of central region should be understood as a first canonical approximation rather than as a final exhaustive complement of all extremal zones.

\subsection{The spine}

We now turn to the axial skeletal structure connecting the self-conjugate vertices.

If \(\SC_n\neq\varnothing\), write
\[
\SC_n=\{\sigma_1,\sigma_2,\dots,\sigma_m\}
\]
in decreasing lexicographic order, that is, from largest to smallest with respect to the lexicographic order on partitions: \(\sigma_i=(\sigma_{i,1},\sigma_{i,2},\dots)\) precedes \(\sigma_j\) when, at the first index where they differ, one has \(\sigma_{i,k}>\sigma_{j,k}\). For each consecutive pair \((\sigma_i,\sigma_{i+1})\), define
\[
\mathrm{Bridge}(\sigma_i,\sigma_{i+1})
\]
to be the union of the vertex sets of all shortest paths between \(\sigma_i\) and \(\sigma_{i+1}\).

\begin{definition}
If \(\SC_n=\varnothing\), set \(\Spine_n:=\varnothing\). Otherwise, let \(\Spine_n\) be the induced subgraph of \(\Gn\) on the vertex set
\[
V(\Spine_n):=\SC_n\cup\bigcup_{i=1}^{m-1}\mathrm{Bridge}(\sigma_i,\sigma_{i+1}).
\]
\end{definition}

In other words, the spine consists of the self-conjugate axis together with the shortest graph-theoretic corridors connecting consecutive self-conjugate vertices in a fixed canonical order. This makes \(\Spine_n\) a distinguished axial subgraph of \(\Gn\).

The definition is designed to capture the intuitive idea that the spine should consist not only of the fixed points of conjugation, but also of the intermediate vertices through which the axis acquires graph-theoretic thickness and connectivity. In this way, the spine is a canonical skeletal extension of the self-conjugate axis.

At this stage of the theory, the spine should be regarded as a first canonical version of an axial skeleton. It is possible that later work will isolate thinner or more selective variants, such as a minimal spine, a median spine, or a support-adapted spine. For descriptive purposes, however, the definition above is a convenient first canonical construction: it is intrinsic, symmetry-based, computable directly from the graph, and explicit about the ordering convention it uses. We do not claim at this stage that it is the unique intrinsic notion of spine.

\subsection{Zone versus skeleton}

The central region and the spine are closely related, but they are not the same object. The central region is a zone, defined by exclusion of the outer framework and bounded axial distance. The spine is a skeletal subgraph, defined by self-conjugate vertices and shortest bridges between them. Thus the former is volumetric, while the latter is corridor-like and structural.

This distinction is important morphologically. A central region may contain vertices that are near the axis but play no privileged connecting role. Conversely, the spine singles out vertices that participate in the axial linkage of the graph, even when the surrounding central zone is broader. One may think of the spine as an internal organizing scaffold inside the central body.

The atlas suggests that these two notions are geometrically meaningful. For small and moderate values of \(n\), the narrow central region and its first thickening already isolate the visually central part of the graph, while the spine remains concentrated near the self-conjugate axis rather than spreading over the whole interior. This is the kind of behavior one would expect from a first axial skeleton in a growing discrete geometric object.

The role of these notions in the overall program is twofold. First, they provide a language for locating higher local complexity: larger degrees and higher simplex layers tend to concentrate not just in the abstract interior, but specifically in the vicinity of the axial central zone. Second, they prepare a more refined theory of centrality and anisotropy. Once one has a central region and a spine, it becomes meaningful to ask how local mobility, simplex richness, and future support-based invariants are distributed relative to them.

For the purposes of the paper, the main point is that the interior of \(\Gn\) is not just the complement of its outer framework. It possesses its own internal organization. The central region captures this organization at the level of zones, while the spine captures it at the level of axial skeletal structure.

\section{Directionality and anisotropy}

The partition graph \(\Gn\) is not morphologically isotropic. Although it is defined purely combinatorially, its large-scale organization singles out several distinct types of directions. These directions are not introduced through any external embedding, but arise from the internal geometry of the graph itself: from the two antenna vertices, the main chain connecting them, the edges, the self-conjugate axis, and the central region around that axis.

\subsection{Longitudinal and lateral directions}

The first and most obvious direction is the \emph{longitudinal direction} determined by the main chain
\[
\Mn=\{(n-k,1^k):0\le k\le n-1\}.
\]
This chain connects the two antenna vertices and provides the basic end-to-end orientation of the graph. Moving along the main chain corresponds to successive elongation of the first column and shortening of the first row, or vice versa. In this way, the main chain defines the primary longitudinal axis of the outer framework.

A second family of directions is provided by the two lateral edges
\[
\Ln \qquad\text{and}\qquad \Rn,
\]
which branch from the main chain near its two ends. These directions are not equivalent to the longitudinal direction: they do not connect the antenna vertices, but rather lead away from the main chain toward side extremal regimes. Thus the outer framework already distinguishes between motion along the main axis and lateral motion away from it.

\subsection{Axial and transverse moves}

A third directionality is induced by conjugation symmetry. Since the self-conjugate axis
\[
\SC_n=\{\lambda\vdash n:\lambda=\lambda'\}
\]
plays the role of a canonical symmetric center, it is natural to regard moves according to their effect on the axial distance
\[
\axdist(\lambda)=\dist_{\Gn}(\lambda,\SC_n).
\]
This leads to a basic axial dichotomy.

\begin{definition}
An oriented edge \(\lambda\to\mu\) in \(\Gn\) is called
\begin{itemize}[leftmargin=2em]
\item \emph{axial} if \(\axdist(\mu)\le \axdist(\lambda)\),
\item \emph{transverse} if \(\axdist(\mu)>\axdist(\lambda)\).
\end{itemize}
\end{definition}

Thus axial moves lead toward the self-conjugate zone or remain within it, whereas transverse moves lead away from the axis toward the outer parts of the graph.

This terminology should be understood as a first morphological approximation rather than as a complete directional calculus. The same edge may be axial when viewed from one endpoint and transverse when viewed from the other, and a more refined directional theory would likely distinguish several kinds of transverse behavior. Nevertheless, the axial/transverse distinction already captures an important large-scale fact: not all moves in the graph play the same geometric role.

In particular, the directionality induced by the self-conjugate axis is different from the longitudinal direction induced by the main chain. A move may be longitudinal without being strongly axial, or axial without lying on the main chain. These two organizing principles therefore intersect but do not coincide. The main chain describes the end-to-end outer morphology of the graph, whereas axial distance describes the position of a vertex relative to the symmetric center.

\subsection{Spinal direction}

The existence of these distinct directional regimes means that the graph is not isotropic. This motivates the use of the term \emph{anisotropy}. The graph \(\Gn\) has multiple distinguished directional regimes:
\begin{itemize}[leftmargin=2em]
\item motion along the main chain;
\item motion along the left or right edge;
\item motion toward or away from the self-conjugate axis;
\item motion inside the central region and along the spine.
\end{itemize}
These regimes are not interchangeable at the level of description. They provide a language for asking whether the graph exhibits different combinatorial behavior in different directions.

The degree landscape and the simplex layers motivate this viewpoint. In the computed range, higher local mobility and richer local simplex structure tend to accumulate in the central body near the axis, whereas the outer framework remains more rigid and constrained. Thus the directional language introduced here is intended as a framework for later quantitative study rather than as a completed directional theory.

The spine introduces an additional refinement. Since it is built from the self-conjugate axis together with shortest bridges between neighboring self-conjugate vertices, it defines a preferred internal corridor system. Movement along the spine should be regarded as a distinguished type of axial motion: not just motion toward the center, but motion along the central skeletal scaffold itself.

For the descriptive purposes of the paper, we do not attempt to formalize all possible directional types. In particular, we do not yet separate different kinds of transverse motion, nor do we formalize a rear direction associated with the still provisional rectangular contour. What matters here is the broader conclusion: the partition graph is not directionally homogeneous. Its morphology singles out a small number of canonical directional tendencies, and these directional tendencies are closely tied to the emergence of center, framework, and spine.

This anisotropy is important for the overall program of the paper. It suggests that future invariants of partition graphs should not be studied only as scalar quantities attached to vertices, but also in terms of how they vary along distinguished directions. Degree, simplex complexity, support phenomena, and jump-type invariants may all behave differently along the longitudinal, lateral, axial, and spinal directions. The present section is only a first step toward a more systematic directional geometry of \(\Gn\).

\section{Growth and morphogenesis}

The preceding sections described the partition graph \(\Gn\) for fixed \(n\): its outer framework, its conjugation symmetry, its local simplex layers, its degree landscape, and its emerging central region and spine. We now turn to the family
\[
G_1,G_2,G_3,\dots
\]
as a whole.

\subsection{Comparative growth}

The guiding viewpoint of this section is morphogenetic. We are interested not only in the structure of an individual graph \(\Gn\), but in the way successive partition graphs exhibit the emergence of new morphological regimes as \(n\) varies. Growth is understood here comparatively. We do not impose a single fixed filtration
\[
\Gn\hookrightarrow G_{n+1}.
\]
At the same time, natural embeddings or overlays of smaller partition graphs into larger ones may well exist and deserve separate study; we do not attempt to formalize them here. What matters here is the comparative perspective: we study how structural features appear, separate, thicken, and interact across the family as \(n\) increases.

This viewpoint is already compatible with the broader program announced in the local paper \cite{LyuLocal}, where the first local layer was developed as a basis for the later study of larger-scale structures, central regions, core-like subgraphs, and growth phenomena as \(n\) varies.

\subsection{Broadening of local regimes}

A first aspect of morphogenesis is the emergence of new local regimes. The local theory shows that degree, neighborhood type, local clique number, and local simplex dimension are controlled by the ordered local transfer type of a partition. As \(n\) grows, the repertoire of such local transfer types becomes richer, and with it the variety of local environments that may occur inside \(\Gn\). Thus growth is not just an increase in the number of vertices, but also an increase in the diversity of local combinatorial patterns available to the graph.

A second aspect is the separation of large-scale structural roles. In the smallest graphs, the distinction between outer framework, central zone, and axial structure is still weak. As \(n\) increases, however, the antenna vertices, the main chain, the edges, the self-conjugate axis, and the central body become increasingly distinguishable from one another in the atlas. What is only a faint organizational tendency in the smallest cases becomes, in moderate values of \(n\), a visible large-scale architecture.

This can be tracked through several observable spectra. For example, one may consider the degree spectrum
\[
\Sigma_n^{\deg}:=\{\deg_{\Gn}(\lambda):\lambda\in V(\Gn)\}
\]
and the simplex spectrum
\[
\Sigma_n^{\mathrm{simp}}:=\{\locdim(\lambda):\lambda\in V(\Gn)\}.
\]
In the computed atlas of small partition graphs, both spectra broaden with \(n\). This means that growth is accompanied by the appearance of new levels of local mobility and new levels of local simplex richness. Morphologically, the graph develops a more articulated internal relief.

\subsection{Centralization and thickening}

A third aspect of morphogenesis is centralization. The atlas suggests that the richer local regimes do not appear uniformly throughout the graph. Rather, they tend to accumulate nearer the self-conjugate axis and inside the central body, while the outer framework remains comparatively constrained. Growth is not isotropic expansion in this sense. It is accompanied by the formation of a progressively more structured interior.

The same phenomenon can be described in terms of thickening. The narrow central region
\[
\Cn{1}
\]
and its first enlargement
\[
\Cn{2}
\]
become increasingly meaningful as \(n\) grows: they isolate not merely a few exceptional vertices, but a genuinely internal zone of the graph. Likewise, the spine becomes progressively more visible as an axial scaffold rather than as a purely formal definition. Thus the central morphology of \(\Gn\) is not present from the outset in full form; it emerges gradually.

A fourth aspect of morphogenesis is the stabilization of contrast between local and global complexity. On the one hand, the local simplex structure becomes richer and richer: higher-dimensional local clique patterns appear, and the degree landscape becomes more varied. On the other hand, the earlier topological work \cite{LyuHomotopy} shows that the global homotopy type of the clique complex remains strikingly simple: for every \(n\), \(\Kn=\Cl(\Gn)\) is homotopy equivalent to a wedge of \(2\)-spheres. The family \((G_n)\) therefore exhibits a characteristic tension between growing local complexity and stable global topological simplicity.

\subsection{Emergence thresholds}

It is useful to isolate the notion of first appearance.

\begin{definition}
Let \(\Phi\) be a graph-theoretic or morphological feature of partition graphs. If the set
\[
\{n\ge 1: G_n \text{ has feature } \Phi\}
\]
is nonempty, then its least element is called the \emph{emergence threshold} of \(\Phi\).
\end{definition}

This language is intentionally broad. A feature \(\Phi\) may be the existence of a given simplex layer, the appearance of a new degree value, the nonemptiness of a central region of a certain kind, the appearance of a nontrivial spine bridge, or the first occurrence of a particular local transfer type. The point is not that all such thresholds should be determined here, but that morphogenesis in the family \((G_n)\) can be studied systematically through them.

From this point of view, the small-graph atlas should be read not just as a collection of examples, but as a first developmental chart. It shows that several phenomena arise in stages. First come the most rigid extremal features: antenna vertices and the hook-like main chain. Then lateral differentiation becomes visible through the edges. After that, richer local simplex behavior and a nontrivial central body begin to emerge. Finally, the axial scaffold represented by the spine becomes identifiable as a distinct structural object inside the interior.

Of course, not every morphological quantity is monotone in a strict sense as \(n\) varies. The family of partition graphs is too intricate for such a naive picture. The claim is not that every observable grows monotonically, but that the global repertoire of structural regimes becomes richer. New local types appear, the interior becomes more articulated, and the contrast between center and periphery becomes more pronounced.

This is the sense in which the partition graph should be viewed as a growing discrete geometric object. Its growth is not merely numerical. It also involves the emergence of new zones, new directional tendencies, new local regimes, and new internal scaffolds across the family.

For the purposes of the paper, the main role of this morphogenetic viewpoint is organizational. It allows us to regard the structures introduced earlier---framework, axis, simplex layers, degree landscape, central region, and spine---as successive ingredients in a developmental picture of the family \((G_n)\).

\section{Atlas of small partition graphs}

The structures introduced in the preceding sections---antenna vertices, main chain, left and right edges, boundary framework, self-conjugate axis, simplex layers, degree landscape, central region, and spine---are most easily understood not in the abstract, but through the first members of the family
\[
G_1,G_2,\dots,G_{12}.
\]
We therefore include a small computational atlas of partition graphs. Its purpose is illustrative and exploratory rather than evidential in any asymptotic sense. It records the first visible morphological patterns in the range \(1\le n\le 12\), and it serves as a compact reference for the structures introduced above. For readability, each atlas is presented as a short figure series, with separate pages for the ranges \(1\le n\le 4\), \(5\le n\le 8\), and \(9\le n\le 12\), together with a larger focused view of \(G_{12}\).

\subsection{Purpose of the atlas}

The structural atlas series displays the outer framework of each graph: the antenna vertices, the hook-like main chain, the left and right edges, and the position of the self-conjugate vertices. The degree atlas series shows how local mobility is distributed in the computed range. The simplex-layer atlas series shows where richer local clique structure first becomes visible. Finally, the central-region/spine overlay series and the larger focused view of \(G_{12}\) display the emergence of an internal organization distinct from the outer framework. None of these figures is intended as a substitute for proof or as asymptotic evidence; their role is to make the small-range geometry legible and to motivate later quantitative questions.

At the smallest values of \(n\), the graphs are too small for a clear separation of morphological roles. The graph is still dominated by its most rigid outer features, and the distinction between framework, center, and spine is weak or degenerate. Nevertheless, even in these earliest cases the essential ingredients are already present in embryonic form: the two antenna vertices, the hook-like longitudinal chain between them, and the conjugation symmetry exchanging the two ends.

As \(n\) increases, the boundary framework becomes easier to recognize. The main chain acquires visible length, and the two side edges begin to separate from it as distinct lateral structures. At this stage one can already speak of an outer framework rather than just a path with small local perturbations.

\subsection{How to read the atlas}

All atlas figures use the same plotting convention. The dashed vertical line marks the locus \(\lambda_1-\ell(\lambda)=0\), so conjugate partitions appear on opposite sides of this axis and self-conjugate partitions lie on it. The remaining vertical offsets are chosen only to separate vertices and edges and to make local structure legible. Thus the atlas should be read as a graph-theoretic picture with a fixed left/center/right organization, not as a claim about intrinsic Euclidean coordinates.

A further stage is marked by the appearance of visibly richer local regimes in the interior. The degree atlas and the simplex-layer atlas show that local complexity does not spread uniformly across the graph. Instead, larger degrees and higher simplex layers begin to appear preferentially away from the extremal framework and nearer the self-conjugate zone. This is the first clear sign that the graph has developed an internal body rather than merely an outer shell.

The next stage is axial consolidation. Once the central body becomes distinguishable, the self-conjugate vertices cease to look like isolated symmetric curiosities and begin to act as the organizing core of an interior zone. The axial central regions
\[
\Cn{1},\qquad \Cn{2}
\]
then become visually meaningful, and the chosen canonical spine becomes visually aligned with a nontrivial axial scaffold rather than appearing as an arbitrary graph-theoretic subset.

In this way, the atlas suggests a developmental sequence:
\begin{enumerate}[label=\arabic*.,leftmargin=2em]
\item \emph{antennal phase}: the graph is dominated by the two extremal endpoints and the shortest hook-like connection between them;
\item \emph{framework phase}: the main chain and the side edges become visibly distinct as an outer boundary framework;
\item \emph{interior differentiation phase}: degree and simplex-layer variation begin to concentrate away from the outer framework;
\item \emph{axial consolidation phase}: the central region and the spine emerge as identifiable internal structures.
\end{enumerate}
These phases are not intended as sharply separated theorems, nor do they imply strict monotonicity for every observable. Rather, they are a descriptive way of reading the early morphology of the family.

\subsection{Small-range observations}

The atlas is also useful methodologically. Many of the notions introduced earlier in the paper are intentionally structural and only partly rigid at this stage. The atlas allows one to test whether these notions are visually and conceptually meaningful.

\begin{observation}
In the computed range \(1\le n\le 12\), visual inspection of the atlas suggests that the distinction between outer framework and interior becomes progressively clearer, and that higher local complexity tends to concentrate nearer the axial center than near the extremal framework.
\end{observation}

Concretely, within this computed range, the atlas makes the following small-range features visible:
\begin{itemize}[leftmargin=2em]
\item the distinction between outer framework and interior is already useful in the first twelve cases;
\item higher local complexity often appears nearer the axial center than near the extremal framework;
\item the central region is not simply the complement of the framework, but a visibly coherent zone;
\item the chosen canonical spine is visually aligned, in the computed range, with an axial corridor system rather than scattered across the whole interior.
\end{itemize}

It is worth emphasizing that the atlas is not meant to replace formal arguments. Rather, it plays the role of a structural reference system. The morphology of partition graphs becomes considerably easier to describe once one has a developmental gallery of the first cases. In that sense, the atlas is part of the present foundational picture: it is the small-range testing ground in which the language of the paper acquires its first concrete use.

The figures accompanying this section should be read comparatively rather than individually. The main mathematical content of the atlas does not lie in any single small graph, but in the visible progression across the family. What matters is not only what \(G_8\) or \(G_{12}\) looks like in isolation, but how the sequence
\[
G_1,G_2,\dots,G_{12}
\]
reveals the progressive articulation of framework, center, and spine.

For this reason, the atlas may be regarded as the first computational layer of the morphogenetic program. It does not resolve structural questions by itself, but it identifies the objects that later work should measure more precisely.

\clearpage

\subsection{Figures}
\medskip

\begin{center}
\includegraphics[width=\textwidth,height=.80\textheight,keepaspectratio]{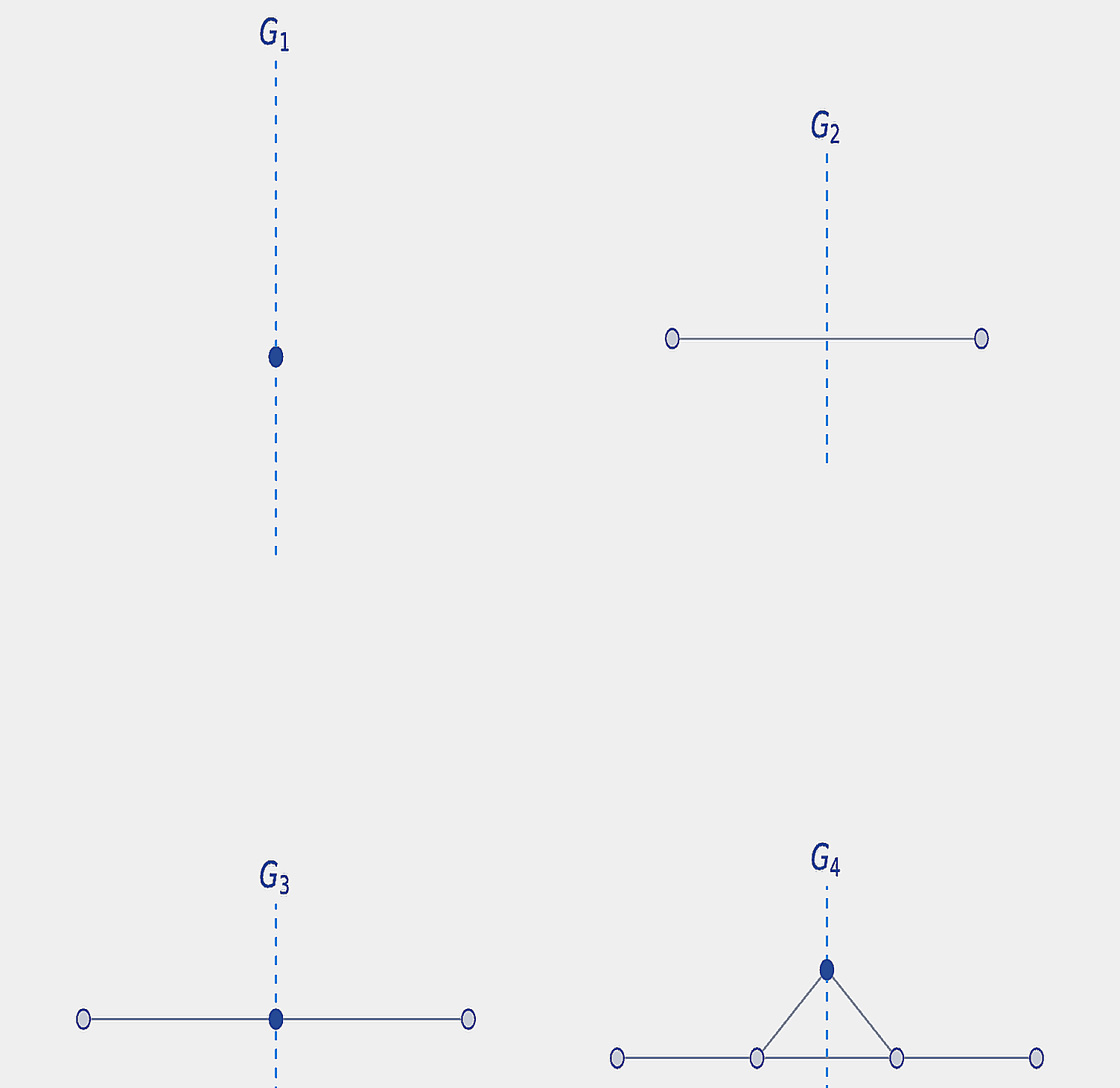}\par\medskip
\captionof{figure}{Structural atlas, Part I: \(G_1,\dots,G_4\). Boundary-framework vertices are lightly tinted, interior vertices are white, self-conjugate vertices are dark blue, and the dashed line marks the self-conjugate axis.}
\label{fig:atlas-structure-1}
\end{center}

\begin{figure}[p]
\centering
\includegraphics[width=\textwidth,height=.83\textheight,keepaspectratio]{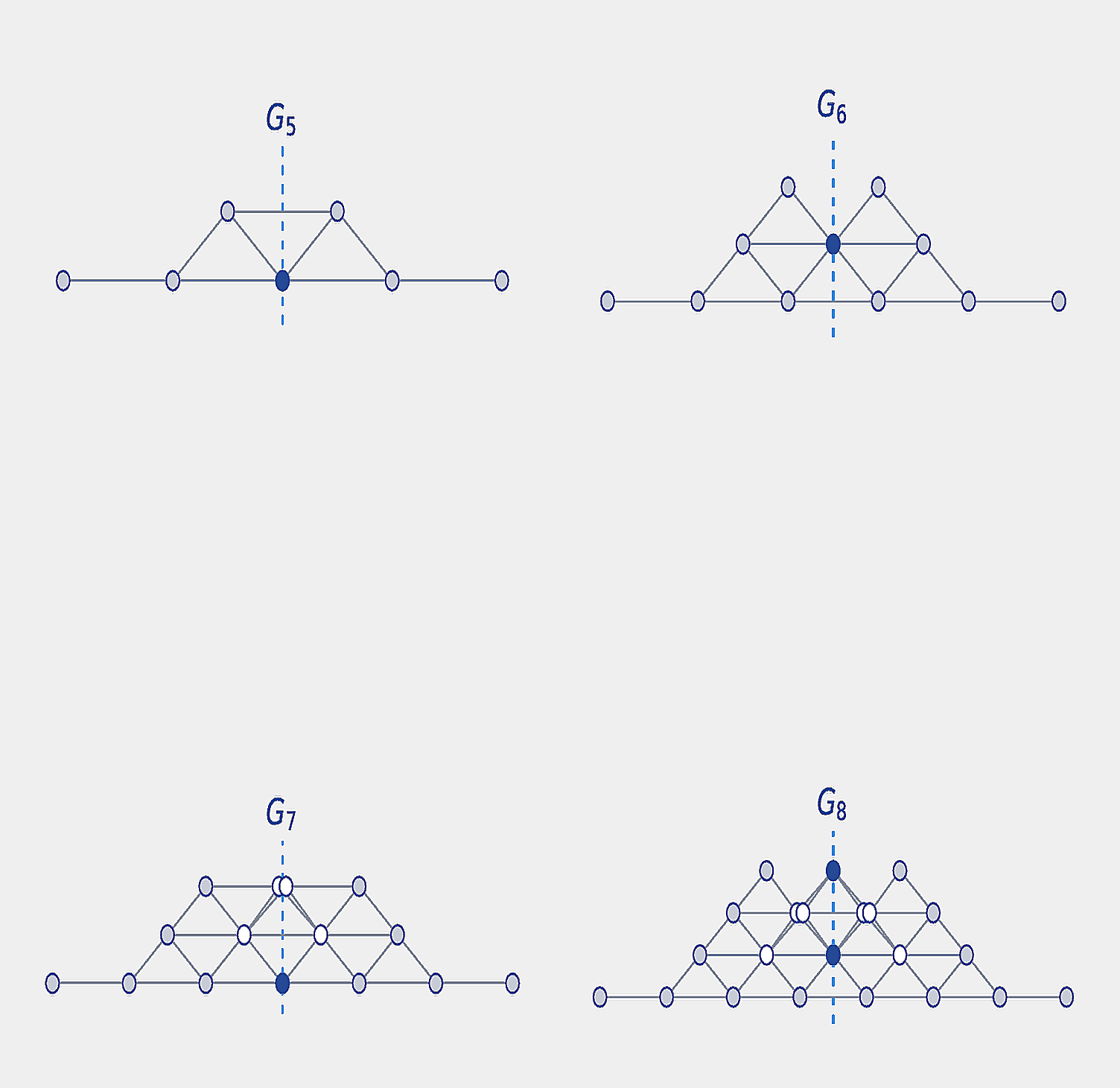}
\caption{Structural atlas, Part II: \(G_5,\dots,G_8\).}
\label{fig:atlas-structure-2}
\end{figure}

\begin{figure}[p]
\centering
\includegraphics[width=\textwidth,height=.83\textheight,keepaspectratio]{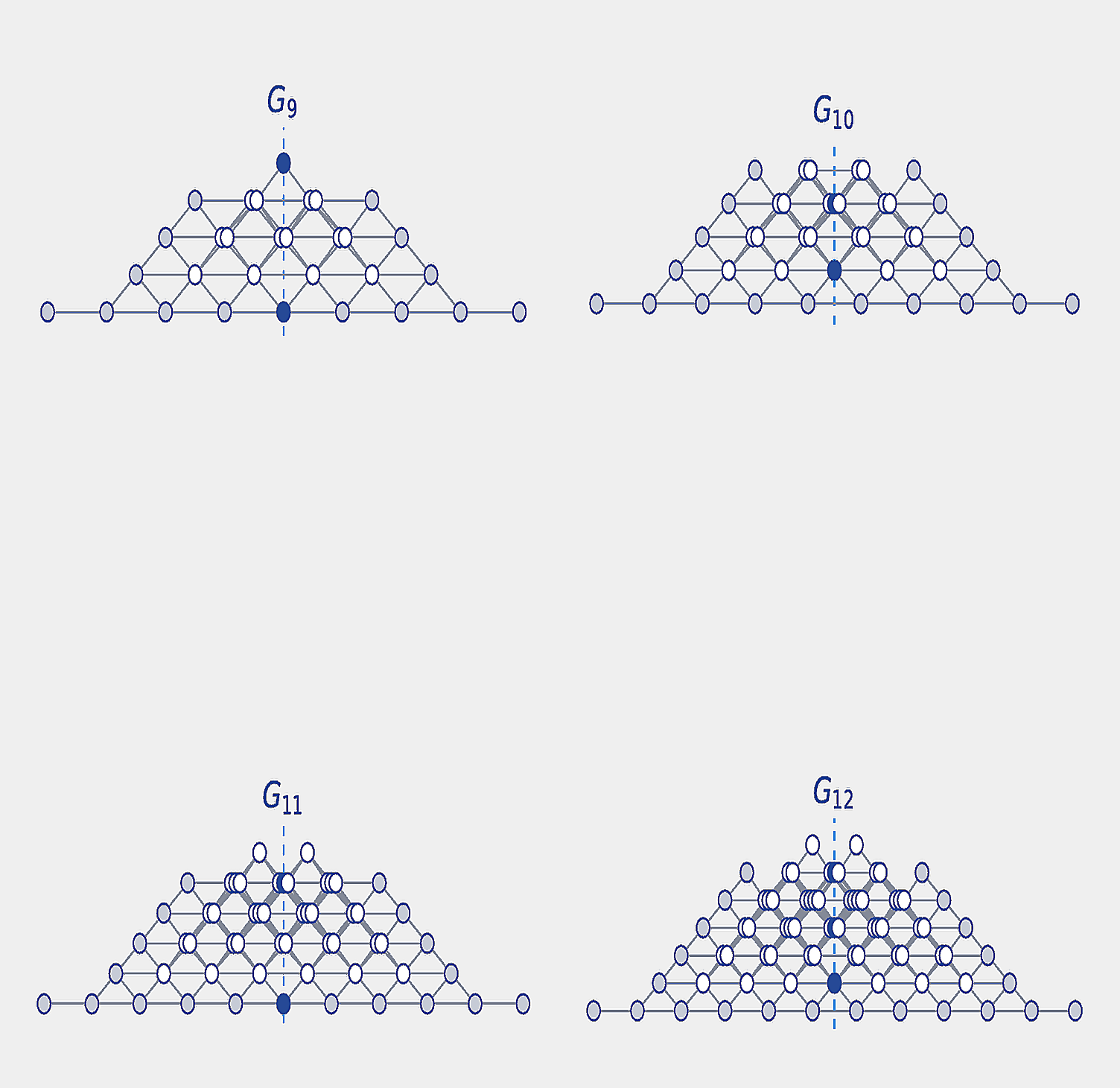}
\caption{Structural atlas, Part III: \(G_9,\dots,G_{12}\).}
\label{fig:atlas-structure-3}
\end{figure}

\begin{figure}[p]
\centering
\includegraphics[width=\textwidth,height=.83\textheight,keepaspectratio]{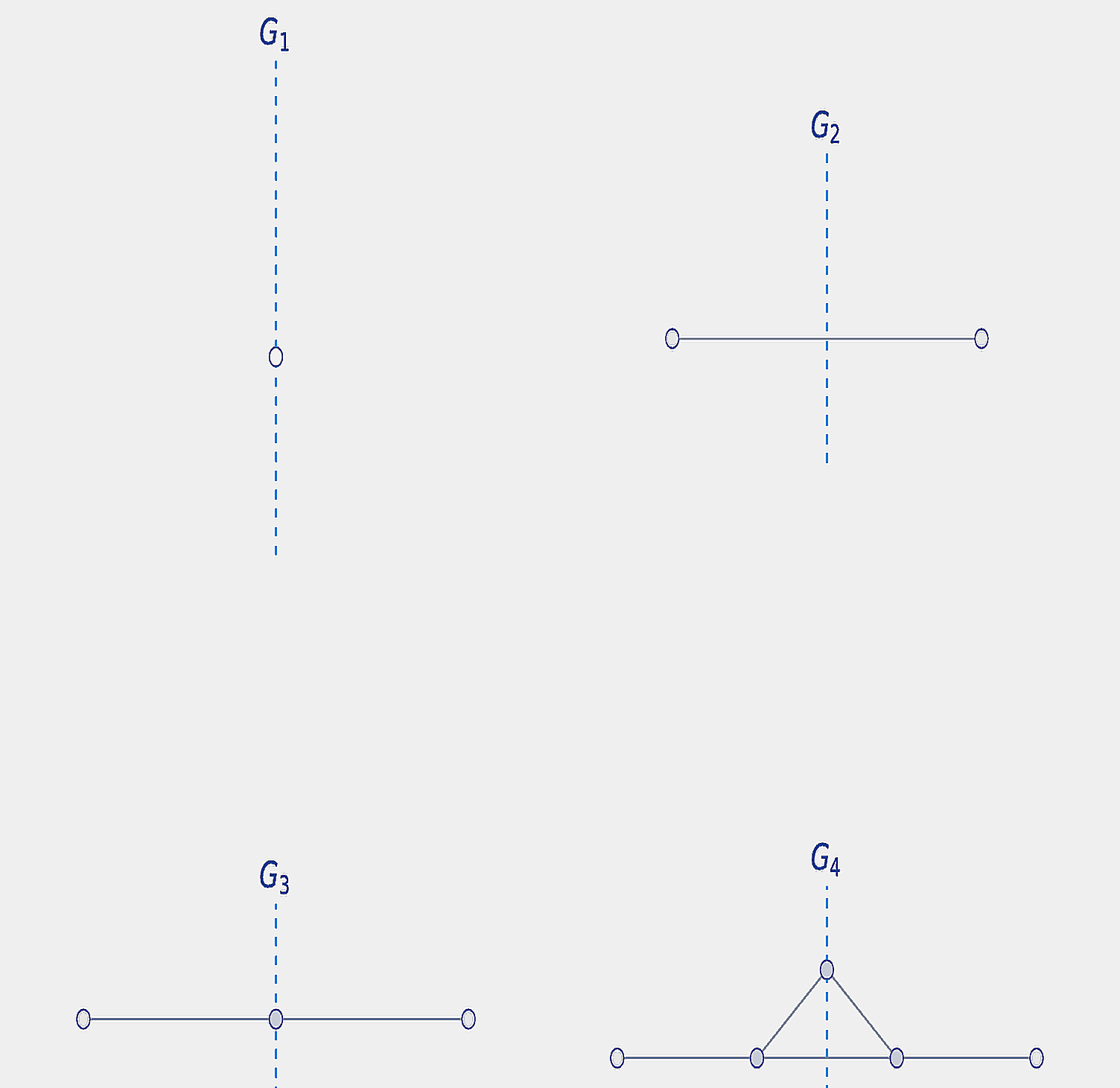}
\caption{Degree atlas, Part I: \(G_1,\dots,G_4\). Node color indicates degree, with a common scale shared across all degree-atlas pages.}
\label{fig:atlas-degree-1}
\end{figure}

\begin{figure}[p]
\centering
\includegraphics[width=\textwidth,height=.83\textheight,keepaspectratio]{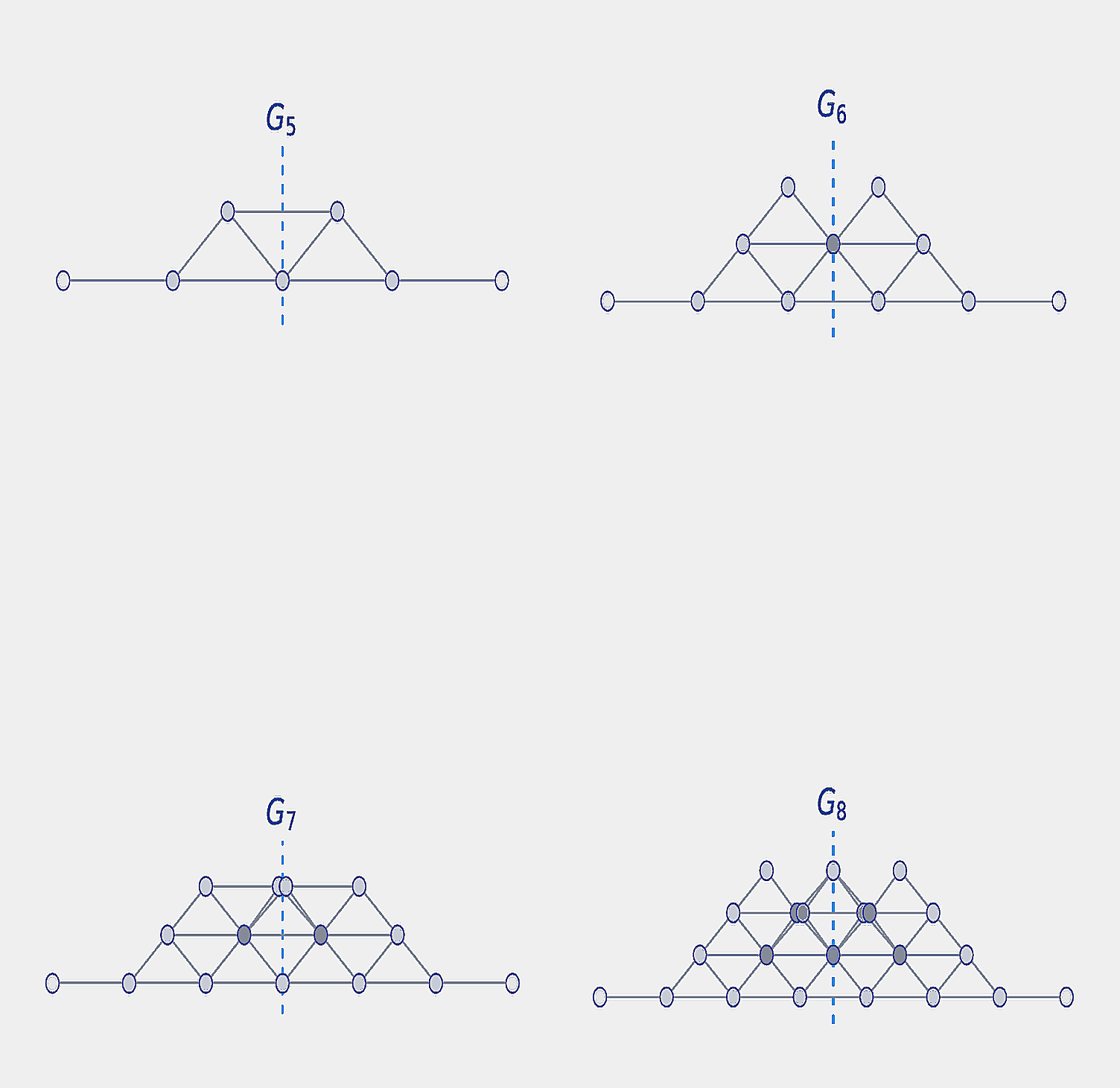}
\caption{Degree atlas, Part II: \(G_5,\dots,G_8\).}
\label{fig:atlas-degree-2}
\end{figure}

\begin{figure}[p]
\centering
\includegraphics[width=\textwidth,height=.83\textheight,keepaspectratio]{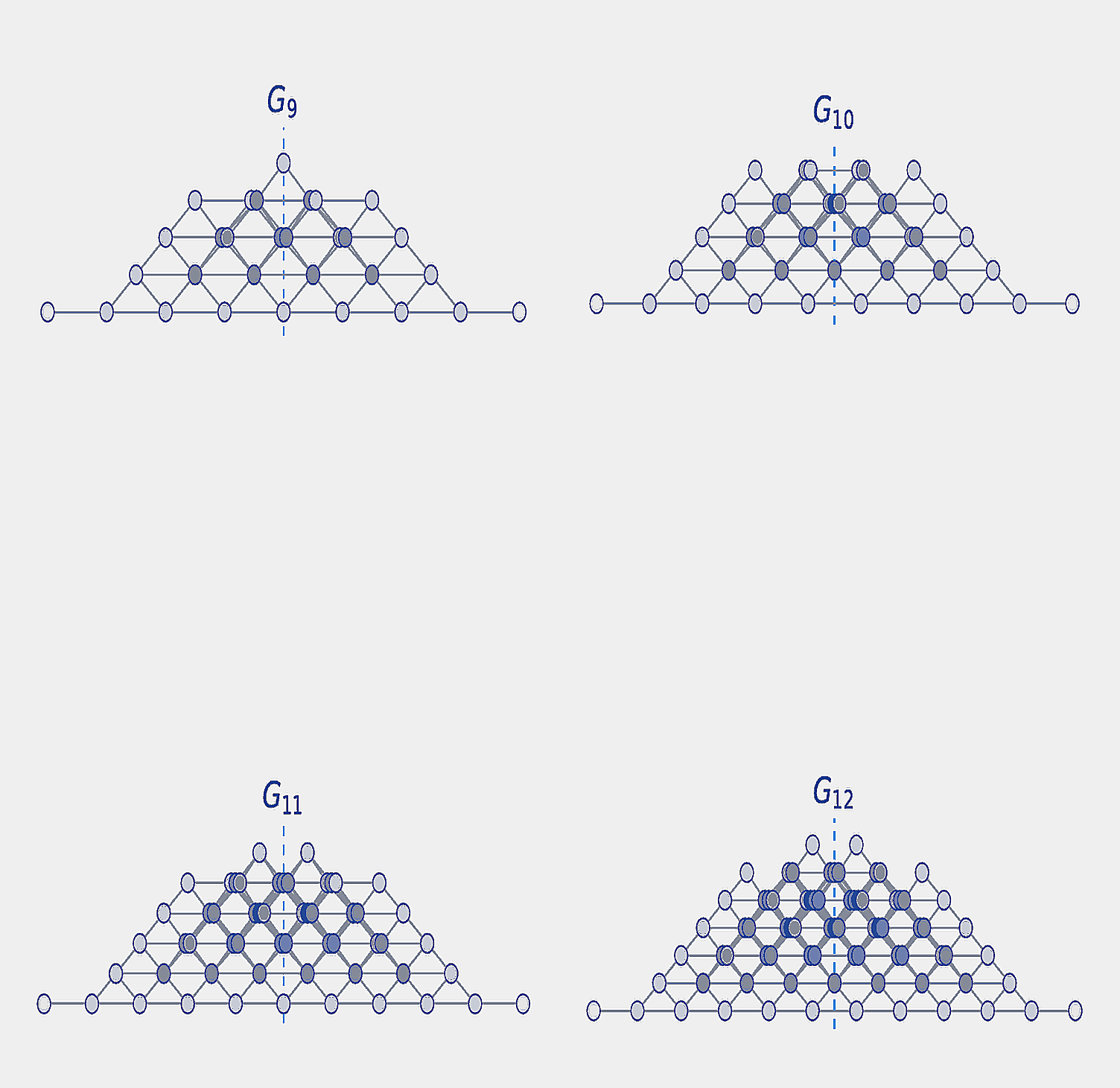}
\caption{Degree atlas, Part III: \(G_9,\dots,G_{12}\).}
\label{fig:atlas-degree-3}
\end{figure}

\begin{figure}[p]
\centering
\includegraphics[width=\textwidth,height=.83\textheight,keepaspectratio]{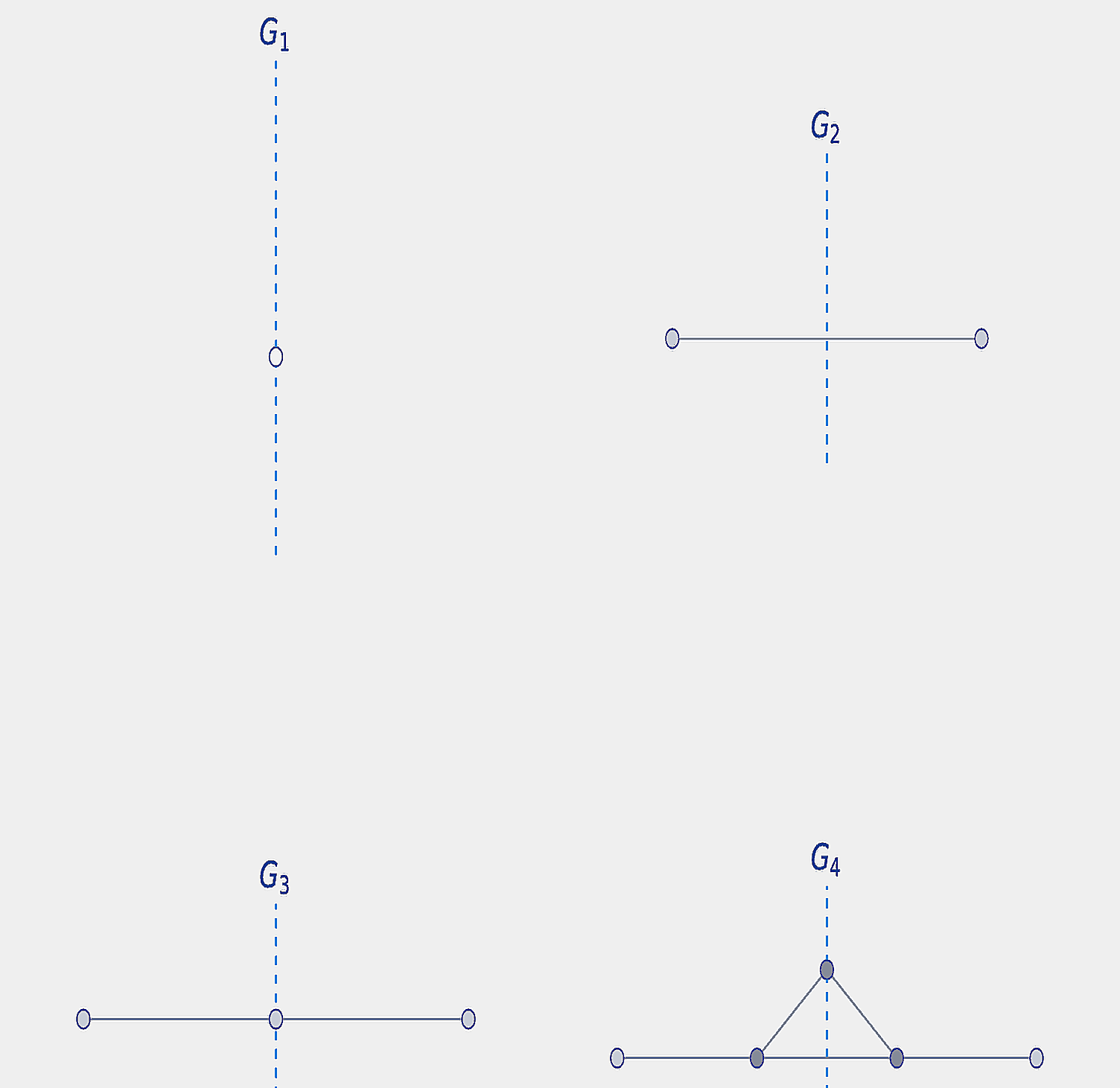}
\caption{Simplex-layer atlas, Part I: \(G_1,\dots,G_4\). Node color indicates local simplex dimension, with a common scale shared across all simplex-atlas pages.}
\label{fig:atlas-simplex-1}
\end{figure}

\begin{figure}[p]
\centering
\includegraphics[width=\textwidth,height=.83\textheight,keepaspectratio]{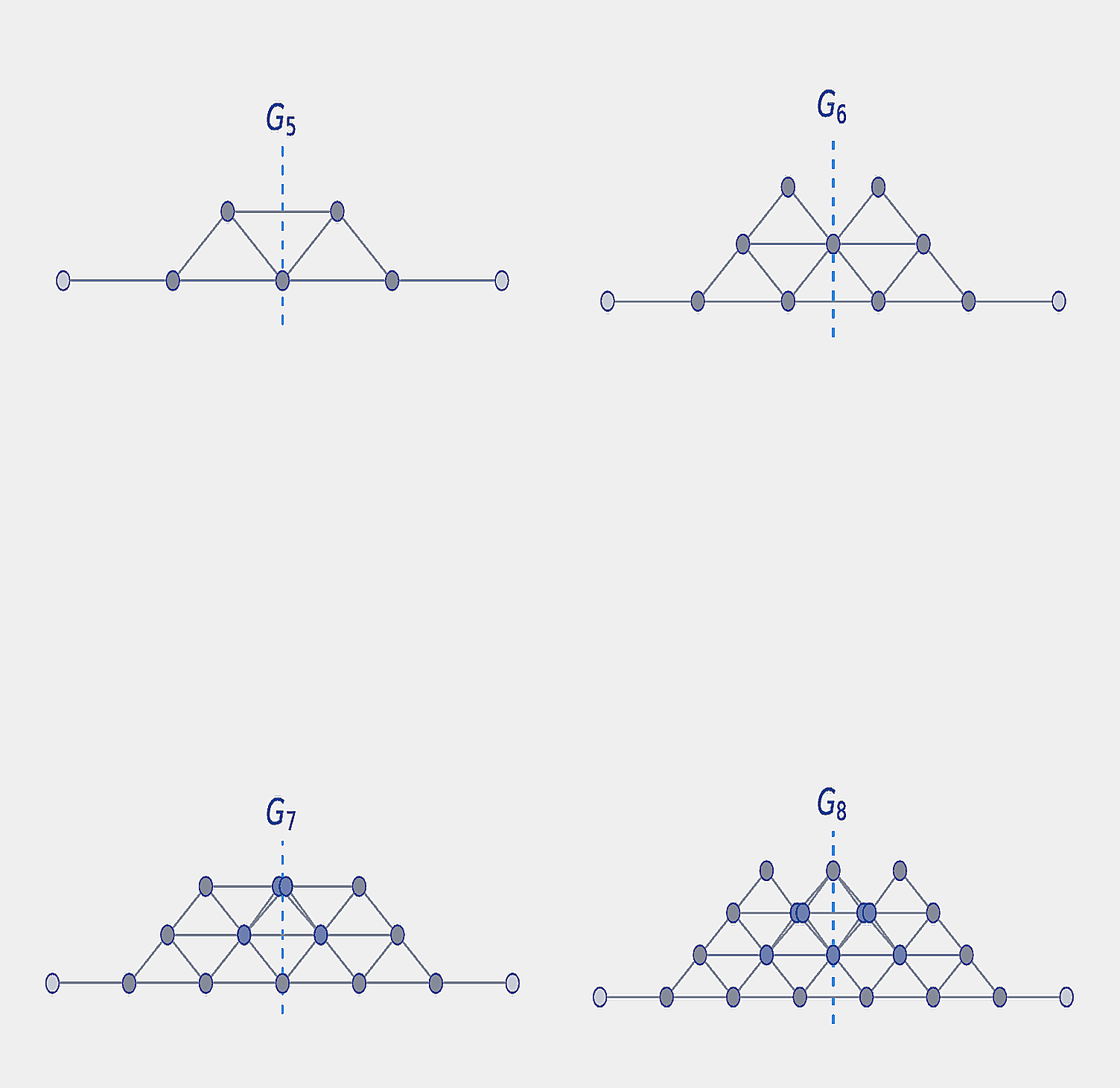}
\caption{Simplex-layer atlas, Part II: \(G_5,\dots,G_8\).}
\label{fig:atlas-simplex-2}
\end{figure}

\begin{figure}[p]
\centering
\includegraphics[width=\textwidth,height=.83\textheight,keepaspectratio]{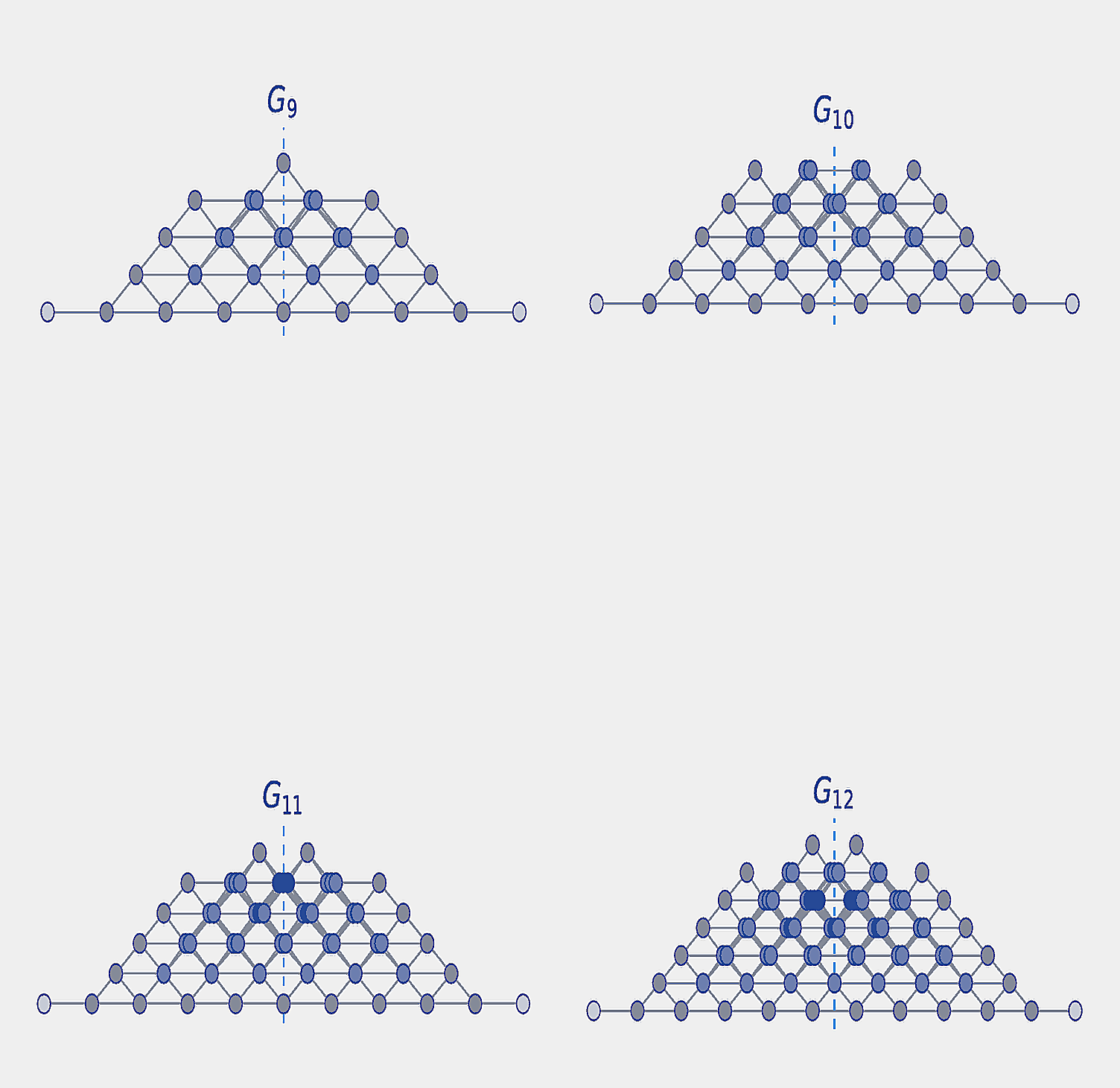}
\caption{Simplex-layer atlas, Part III: \(G_9,\dots,G_{12}\).}
\label{fig:atlas-simplex-3}
\end{figure}

\begin{figure}[p]
\centering
\includegraphics[width=\textwidth,height=.83\textheight,keepaspectratio]{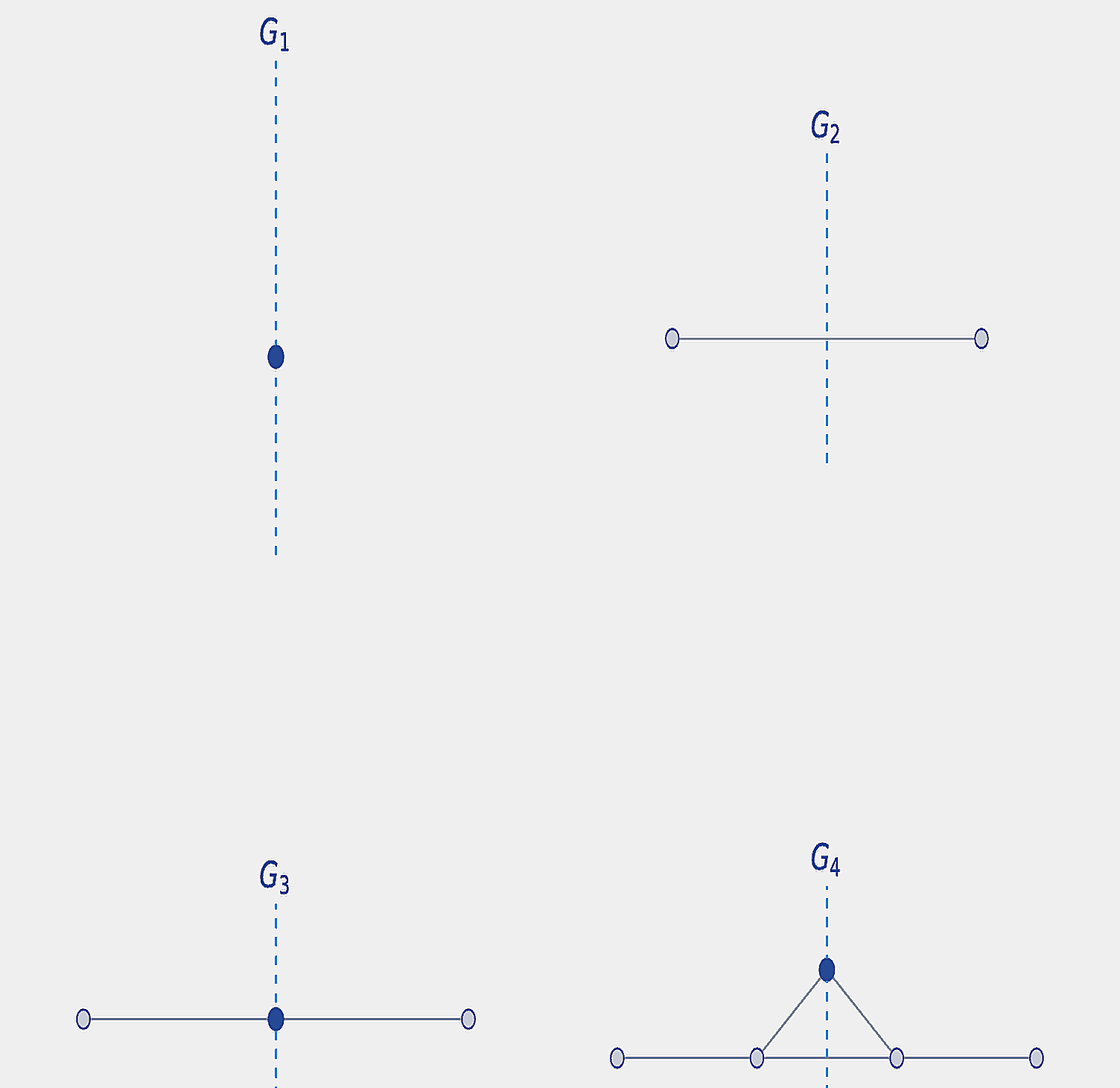}
\caption{Central-region/spine overlay, Part I: \(G_1,\dots,G_4\). Boundary vertices are light gray, \(\mathcal C_n^{(2)}\setminus \mathcal C_n^{(1)}\) is dark gray, \(\mathcal C_n^{(1)}\) is silver, self-conjugate vertices are black, and spine vertices are marked by a thick outline.}
\label{fig:atlas-central-spine-1}
\end{figure}

\begin{figure}[p]
\centering
\includegraphics[width=\textwidth,height=.83\textheight,keepaspectratio]{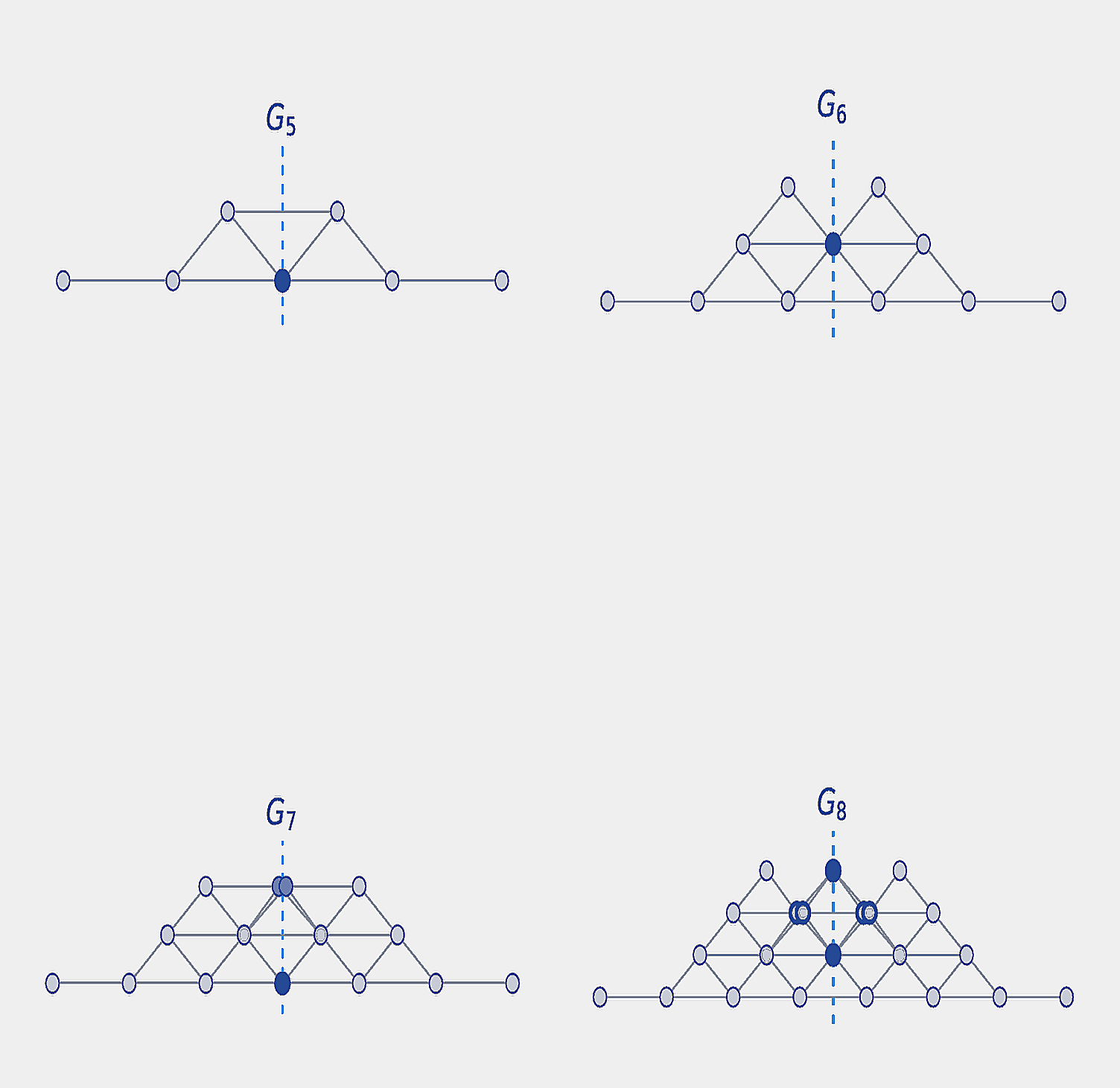}
\caption{Central-region/spine overlay, Part II: \(G_5,\dots,G_8\).}
\label{fig:atlas-central-spine-2}
\end{figure}

\begin{figure}[p]
\centering
\includegraphics[width=\textwidth,height=.83\textheight,keepaspectratio]{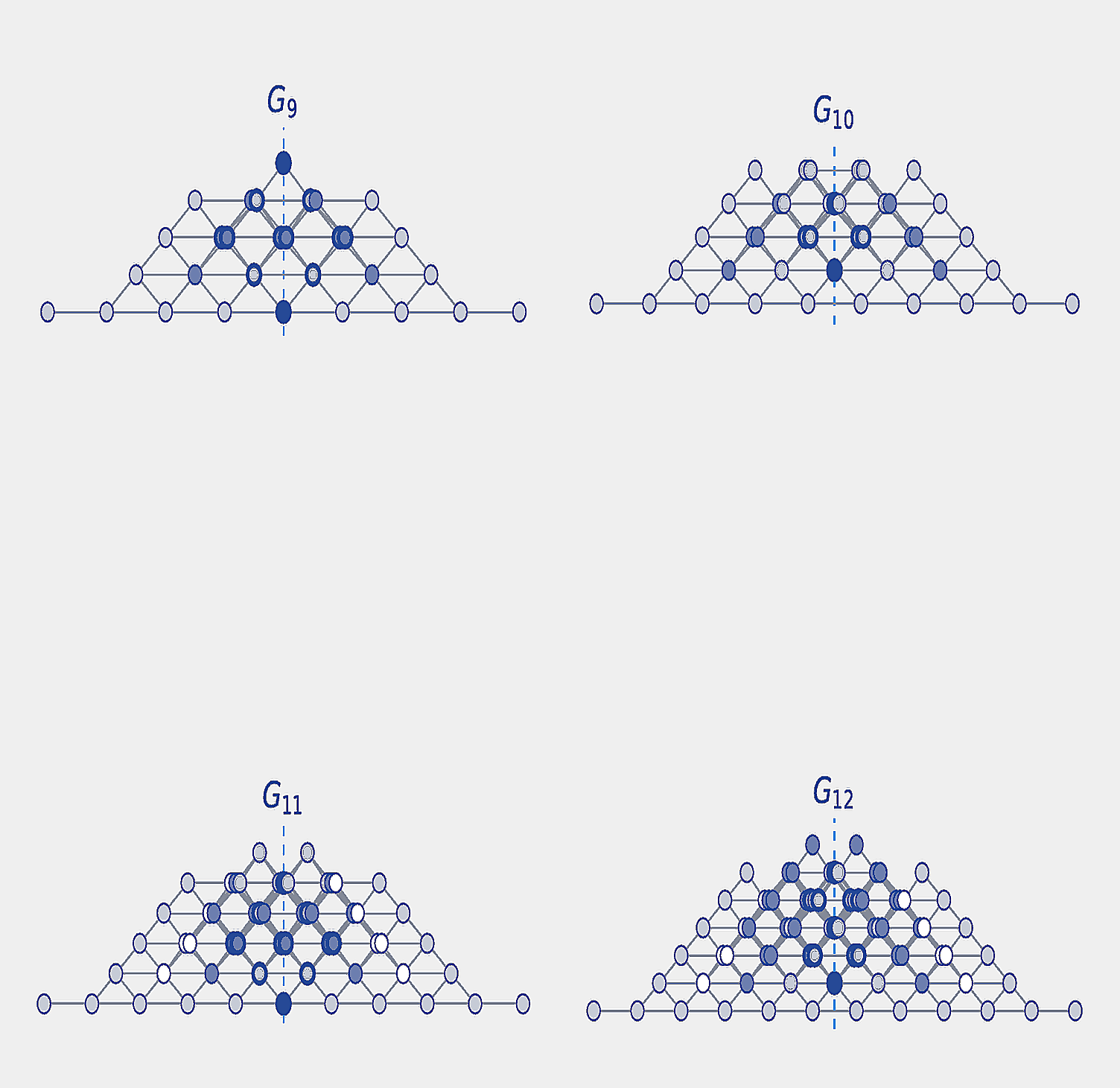}
\caption{Central-region/spine overlay, Part III: \(G_9,\dots,G_{12}\).}
\label{fig:atlas-central-spine-3}
\end{figure}

\begin{figure}[p]
\centering
\includegraphics[width=\textwidth,height=.83\textheight,keepaspectratio]{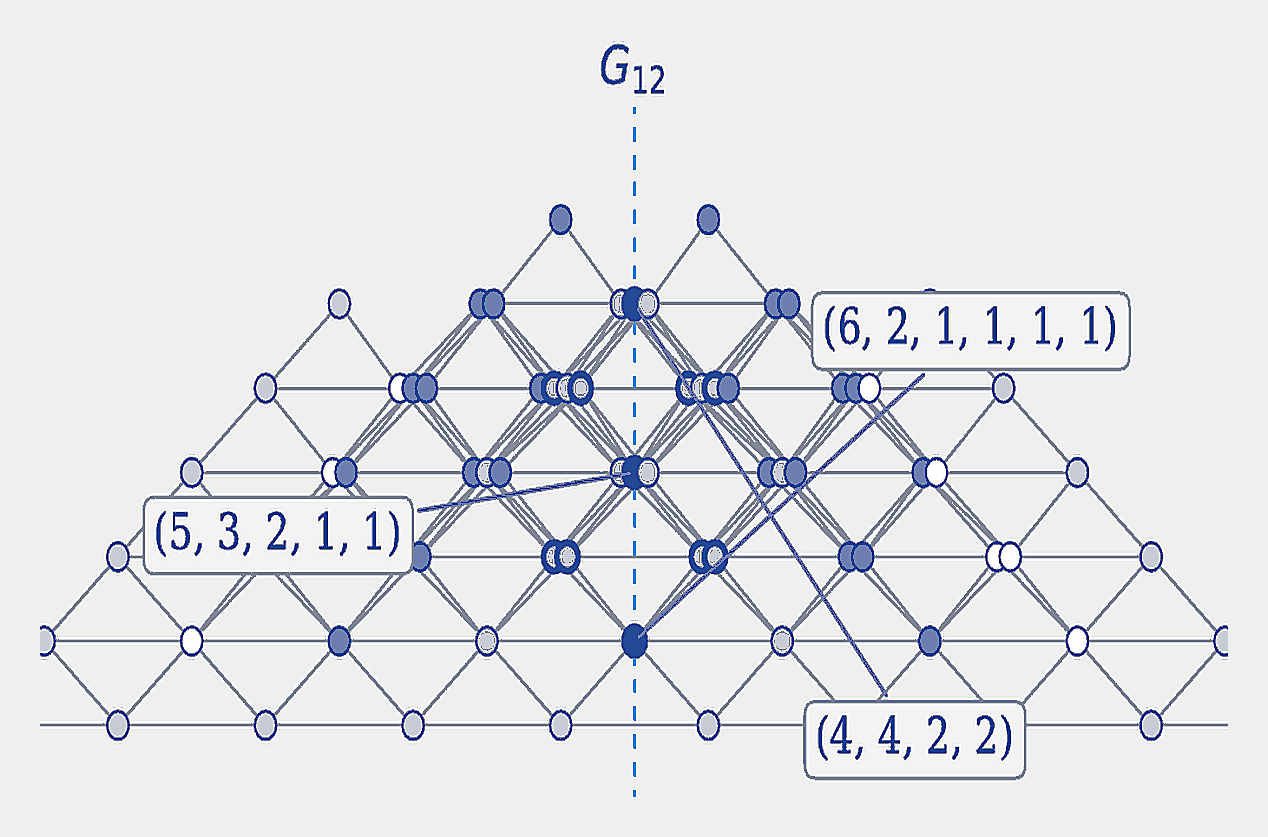}
\caption{Focused view of \(G_{12}\). The three self-conjugate partitions are labeled outside the graph and connected to their vertices by leader lines.}
\label{fig:g12-focus}
\end{figure}

\FloatBarrier

\section{Heuristic principles, open questions, and further directions}

This paper is intended as a structural foundation rather than a final classification of the family. The notions introduced here---framework, axis, simplex layers, degree landscape, central region, and spine---naturally lead to heuristic principles, precise questions, and more focused directions for later work. In this section we isolate the main ones.

\subsection{Framework versus interior}

A first recurring theme of the paper is the distinction between the outer framework and the internal body of the graph. The atlas suggests that this distinction becomes progressively sharper as \(n\) grows. The outer framework remains comparatively rigid and sparse, while the interior becomes increasingly differentiated.

\begin{heuristic}
As \(n\) increases, the contrast between the boundary framework and the internal body of \(\Gn\) should become progressively more pronounced, both in degree distribution and in local simplex layering.
\end{heuristic}

At this stage, this should be read as a guiding structural expectation rather than as a fully formal asymptotic statement. Its eventual precise form may involve degree spectra, simplex spectra, or concentration statements relative to the central regions \(\Cn{r}\).

\subsection{Central concentration}

A second major theme is the apparent concentration of richer local regimes near the self-conjugate axis. In the computed range, higher local simplex layers and larger degrees tend to occur not on the extremal framework, but nearer the axial central zone.

The present framework suggests the following precise questions.

\begin{question}
Does there exist \(R\ge 1\) such that, for all sufficiently large \(n\), every vertex of maximal local simplex dimension belongs to \(\Cn{R}\)?
\end{question}

\begin{question}
Does there exist \(R\ge 1\) such that, for all sufficiently large \(n\), every vertex of maximal degree belongs to \(\Cn{R}\)?
\end{question}

Near-maximal variants of these questions could also be studied, but the formulations above already isolate a first graph-theoretic version of the central-concentration principle. They are recorded here as concrete targets for later work rather than as claims supported by extensive evidence.

\subsection{Spine problems}

The notion of spine introduced in this paper is intentionally canonical but still preliminary. It is defined from the self-conjugate axis and shortest bridges between neighboring self-conjugate vertices in a chosen canonical order. The atlas suggests that this construction is often visually aligned with an axial corridor system of genuine structural interest. At the same time, it is likely not the only meaningful axial skeleton.

\begin{problem}
Determine whether the spine \(\Spine_n\) admits thinner canonical variants, such as a minimal spine, a median spine, or a support-adapted spine.
\end{problem}

\begin{problem}
Study the metric and combinatorial thickness of the spine as \(n\) grows.
\end{problem}

\begin{problem}
Understand how degree, simplex layers, and future support-based invariants are distributed relative to the spine.
\end{problem}

\subsection{Rectangular contour}

The atlas also suggests the presence of a further extremal structure associated with rectangular partitions. Unlike the main chain and the edges, this structure does not yet fit comfortably into the current formal framework. It behaves more like a discontinuous contour than like a simple path.

\begin{problem}
Define and analyze the rectangular contour of \(\Gn\), and determine whether it gives rise to a canonical rear edge or rear boundary layer.
\end{problem}

This direction appears especially promising, since rectangular partitions often occupy a morphologically distinguished position while not belonging to the already formalized outer framework.

\subsection{Embeddings and overlays}

Here growth has been treated comparatively rather than through a single fixed filtration. However, the family of partition graphs appears to admit natural embeddings or overlays of \(\Gn\) into \(G_{n+1}\), and perhaps into larger graphs in multiple non-equivalent ways. These overlays may intersect when repeated, and they seem to encode a nontrivial form of self-similarity across the family.

\begin{problem}
Construct and study natural embeddings or overlays of \(\Gn\) into \(G_{n+1}\) and, more generally, into \(G_{n+k}\).
\end{problem}

\begin{problem}
Analyze how repeated overlays intersect, and whether they induce a meaningful recursive or self-similar structure on the family \((G_n)\).
\end{problem}

\subsection{Future specialized theories}

Here degree and simplex layers are treated only at the level of large-scale morphology. But both are part of a broader local-to-global program already implicit in \cite{LyuLocal}, where the ordered local transfer type determines the neighborhood graph, local clique number, degree, and local simplex dimension.

\begin{problem}
Develop a systematic theory of the degree landscape: spectra, extremal vertices, concentration phenomena, and asymptotic growth.
\end{problem}

\begin{problem}
Develop a theory of support and support jumps: how support size shapes local and global morphology, and how changes in support interact with the directional geometry of the graph.
\end{problem}

\begin{problem}
Develop a theory of jump-type or gradient-type invariants describing local morphological transitions across the graph.
\end{problem}

A final theme concerns the relation between the present morphological picture and the earlier global topological result on the clique complex \(\Kn=\Cl(\Gn)\). The present paper emphasizes the emergence of richer local simplex structure, more articulated degree landscapes, and more visible internal organization as \(n\) grows. Yet \cite{LyuHomotopy} shows that the global homotopy type remains extremely simple: \(\Kn\) is always homotopy equivalent to a wedge of \(2\)-spheres.

\begin{problem}
Explain more conceptually why increasing local and morphological complexity in \(\Gn\) coexists with the stable global topological simplicity of \(\Kn\).
\end{problem}

\subsection{Final perspective}

The paper should be understood as foundational rather than terminal. Its main contribution is not a single maximal theorem, but a structural language together with a first developmental picture. It identifies a family of canonical objects inside the partition graph and argues that these objects are best understood not in isolation, but as parts of an emerging morphology: outer framework, symmetric axis, local simplex layers, degree landscape, central region, and spine.

The next stage of the project is to make this picture sharper. Some parts call for formal asymptotic statements; others for more refined definitions; still others for focused case studies. The aim of the paper is more modest: to provide a vocabulary, a collection of constructions, and a small-range computational record from which later theorem-centered work can proceed.

\section{Concluding remarks}

The paper has proposed a structural language for viewing the partition graph \(\Gn\) as a growing discrete geometric object. Its contribution is not a single extremal theorem, but the introduction of a coherent family of large-scale objects: antenna vertices, main chain, edges, boundary framework, self-conjugate axis, simplex layers, degree landscape, central region, and spine.

Taken together, these notions suggest that partition graphs possess a persistent and increasingly articulated morphology that is not captured either by the global homotopy type of the clique complex alone or by the local theory at a single vertex alone. The small computational atlas records the first visible stages of this morphology in the range \(1\le n\le 12\): the separation of framework and interior, the appearance of richer local regimes near the axial center, and the emergence of an internal skeletal construction.

The main message is that partition graphs form a mathematically natural family with a nontrivial large-scale geometry, and that describing this geometry is itself a natural mathematical problem.

\section*{Acknowledgements}

The author acknowledges the use of ChatGPT (OpenAI) for discussion, structural planning, and editorial assistance during the preparation of this manuscript. All mathematical statements, proofs, computations, and final wording were checked and approved by the author, who takes full responsibility for the contents of the paper.

\FloatBarrier
\appendix

\section{Small-range numerical tables}

The data in this appendix were obtained from the graph model \(\Gn\) for \(1\le n\le 12\) by explicit enumeration of partitions and elementary transfers. They are included as a transparent computational record for the small-range atlas, not as a substitute for later quantitative analysis. The same graph data were then used to compute degrees, local simplex dimensions, central regions, and the spine. In particular, the central-region counts \(|\Cn{1}|\) and \(|\Cn{2}|\) are computed using the boundary framework
\[
\Fn=\Mn\cup \Ln\cup \Rn,
\]
in accordance with the definitions adopted in the main text. For \(n=2\), where \(\SC_2=\varnothing\), we use the convention \(\dist_{\Gn}(\lambda,\varnothing)=\infty\), so the central regions are empty and the spine is empty as well.

\subsection{Basic graph counts}

For \(1\le n\le 12\), we record the number of vertices \(p(n)\), the number of edges \(|E(\Gn)|\), the number of self-conjugate vertices \(|\SC_n|\), and the size of the boundary framework \(|\Fn|\).

\begin{table}[htbp]
\centering
\caption{Basic graph counts for \(G_1,\dots,G_{12}\).}
\label{tab:basic-counts}
\begin{tabular}{ccccc}
\toprule
\(n\) & \(p(n)\) & \(|E(G_n)|\) & \(|\SC_n|\) & \(|F_n|\) \\
\midrule
1  & 1  & 0   & 1 & 1  \\
2  & 2  & 1   & 0 & 2  \\
3  & 3  & 2   & 1 & 3  \\
4  & 5  & 5   & 1 & 5  \\
5  & 7  & 9   & 1 & 7  \\
6  & 11 & 17  & 1 & 10 \\
7  & 15 & 28  & 1 & 11 \\
8  & 22 & 47  & 2 & 14 \\
9  & 30 & 73  & 2 & 15 \\
10 & 42 & 114 & 2 & 18 \\
11 & 56 & 170 & 2 & 19 \\
12 & 77 & 253 & 3 & 22 \\
\bottomrule
\end{tabular}
\end{table}

\subsection{Local-complexity maxima}

For \(1\le n\le 12\), we record the maximal degree and the maximal local simplex dimension.

\begin{table}[htbp]
\centering
\caption{Maximal degree and maximal local simplex dimension for \(G_1,\dots,G_{12}\).}
\label{tab:maxima}
\begin{tabular}{ccc}
\toprule
\(n\) & \(\max_{\lambda\in V(G_n)} \deg_{G_n}(\lambda)\) & \(\max_{\lambda\in V(G_n)} \dim_{\mathrm{loc}}(\lambda)\) \\
\midrule
1  & 0  & 0 \\
2  & 1  & 1 \\
3  & 2  & 1 \\
4  & 3  & 2 \\
5  & 4  & 2 \\
6  & 6  & 2 \\
7  & 7  & 3 \\
8  & 8  & 3 \\
9  & 8  & 3 \\
10 & 12 & 3 \\
11 & 13 & 4 \\
12 & 14 & 4 \\
\bottomrule
\end{tabular}
\end{table}

\subsection{Central-region data}

For \(1\le n\le 12\), we record the sizes of the narrow central region \(\mathcal C_n^{(1)}\), its first thickening \(\mathcal C_n^{(2)}\), the vertex set of the spine \(V(\Spine_n)\), and the full interior complement \(V(G_n)\setminus F_n\).

\begin{table}[htbp]
\centering
\caption{Central-region and spine data for \(G_1,\dots,G_{12}\): \(|\mathcal C_n^{(1)}|\) is the size of the narrow central region, \(|\mathcal C_n^{(2)}|\) the size of its first thickening, \(|V(\Spine_n)|\) the size of the spine, and \(|V(G_n)\setminus F_n|\) the total number of interior vertices.}
\label{tab:central-spine}
\begin{tabular}{ccccc}
\toprule
\(n\) & \(|\mathcal C_n^{(1)}|\) & \(|\mathcal C_n^{(2)}|\) & \(|V(\Spine_n)|\) & \(|V(G_n)\setminus F_n|\) \\
\midrule
1  & 0  & 0  & 1  & 0  \\
2  & 0  & 0  & 0  & 0  \\
3  & 0  & 0  & 1  & 0  \\
4  & 0  & 0  & 1  & 0  \\
5  & 0  & 0  & 1  & 0  \\
6  & 1  & 1  & 1  & 1  \\
7  & 2  & 4  & 1  & 4  \\
8  & 8  & 8  & 6  & 8  \\
9  & 5  & 15 & 12 & 15 \\
10 & 16 & 24 & 6  & 24 \\
11 & 11 & 27 & 12 & 37 \\
12 & 21 & 45 & 11 & 55 \\
\bottomrule
\end{tabular}
\end{table}

\subsection{Optional spectra}

If desired, one may also include the degree spectrum
\[
\Sigma_n^{\deg}:=\{\deg_{\Gn}(\lambda):\lambda\in V(\Gn)\}
\]
and the simplex spectrum
\[
\Sigma_n^{\mathrm{simp}}:=\{\locdim(\lambda):\lambda\in V(\Gn)\}
\]
for \(1\le n\le 12\), or at least for a selected subrange such as \(8\le n\le 12\). We omit these tables here in order to keep the appendix compact.

\section{Glossary of structural terms}

This appendix collects the main structural terms introduced in the paper. The entries are listed in the order of their first appearance rather than alphabetically.

\glossentry{Antenna vertex}{One of the two extremal hook endpoints \((n)\) and \((1^n)\); for \(n\ge 2\) these are precisely the two degree-one vertices of \(\Gn\), while for \(n=1\) they coincide. The terminology is motivated by the two outer antenna-like protrusions visible in the atlas of small partition graphs.}

\glossentry{Main chain}{The hook-partition path
\[
(n),(n-1,1),(n-2,1^2),\dots,(1^n),
\]
which is a canonical shortest path between the two antenna vertices.}

\glossentry{Left edge}{The path of two-part partitions branching from \((n-1,1)\).}

\glossentry{Right edge}{The conjugate path branching from \((2,1^{n-2})\).}

\glossentry{Boundary framework}{The union
\[
\Fn=\Mn\cup \Ln\cup \Rn
\]
of the main chain and the left and right edges.}

\glossentry{Interior vertex}{A vertex of \(\Gn\) not belonging to \(\Fn\).}

\glossentry{Self-conjugate axis}{The set
\[
\SC_n=\{\lambda\vdash n:\lambda=\lambda'\}.
\]}

\glossentry{Axial distance}{The graph distance from \(\lambda\) to the self-conjugate axis, with the convention \(\dist_{\Gn}(\lambda,\varnothing)=\infty\):
\[
\axdist(\lambda)=\dist_{\Gn}(\lambda,\SC_n).
\]}

\glossentry{Simplex layer}{The set
\[
\simpLayer_r(n)=\{\lambda:\locdim(\lambda)=r\}.
\]}

\glossentry{Degree layer}{The set
\[
\deglayer_d(n)=\{\lambda:\deg_{\Gn}(\lambda)=d\}.
\]}

\glossentry{Degree landscape}{The degree function viewed as a field on all vertices of \(\Gn\).}

\glossentry{Axial central region}{The set
\[
\Cn{r}=\{\lambda\in V(\Gn)\setminus \Fn:\axdist(\lambda)\le r\}.
\]}

\glossentry{Narrow central region}{The region \(\Cn{1}\).}

\glossentry{Spine}{The induced subgraph on the self-conjugate axis together with the shortest bridges between consecutive self-conjugate vertices in decreasing lexicographic order; if \(\SC_n=\varnothing\), the spine is empty.}

\glossentry{Morphogenesis}{The comparative study of how morphological regimes emerge across the family \((G_n)\) as \(n\) varies.}

\glossentry{Emergence threshold}{The least \(n\) for which a given structural feature first appears.}

\glossentry{Rectangular contour}{A provisional term for the still-unformalized extremal structure associated with rectangular partitions.}

\glossentry{Overlay / embedding}{A provisional term for possible natural placements of smaller partition graphs inside larger ones.}

\end{document}